\documentclass[11pt]{amsart}
\usepackage[utf8]{inputenc}
\usepackage[english]{babel}
\usepackage[babel=true]{csquotes}

\usepackage{amssymb}
\usepackage{amsmath}
\usepackage{mathrsfs}
\usepackage{accents}
\usepackage{amsthm}
\usepackage{amscd}
\usepackage{hyperref}
\usepackage{enumerate}

\newcounter{intro}
\newtheorem{introthm}[intro]{Theorem}
\newtheorem{thm}{Theorem}

\newtheorem{cor}[thm]{Corollary}

\newtheorem{lem}[thm]{Lemma}
\newtheorem{prop}[thm]{Proposition}

\newtheorem{rmk}[thm]{Remark}

\newcommand{\cD}{\mathcal {D}}\newcommand{\cE}{\mathcal{E}} 
 \newcommand{\cH}{\mathcal{H}} 
 \newcommand{\cK}{\mathcal{K}} \newcommand{\cL}{\mathcal{L}}
\newcommand{\cM}{\mathcal{M}}  \newcommand{\cO}{\mathcal{O}} 
\newcommand{\cP}{\mathcal{P}}  \newcommand{\cR}{\mathcal{R}}
\newcommand{\cS}{\mathcal{S}}   
\newcommand{\cV}{\mathcal{V}}

\newcommand{\To}{\longrightarrow}

\newcommand{\nP}{P}
\newcommand{\nA}{A}
\newcommand{\rmP}{\mathsf{P}}
\newcommand{\rmA}{\mathsf{A}}
\newcommand{\rmB}{\mathsf{B}}
\newcommand{\rmu}{\mathsf{u}}
\newcommand{\rmv}{\mathsf{v}}
\newcommand{\rmf}{\mathsf{f}}
\newcommand{\rma}{\mathsf{a}}
\newcommand{\rmb}{\mathsf{b}}

\newcommand{\ds}{\displaystyle}

 \newcommand{\CC}{{\mathbb C}} 
  
  \newcommand{\NN}{{\mathbb N}}
  \newcommand{\RR}{{\mathbb R}}
  \newcommand{\ZZ}{{\mathbb Z}}
 
\newcommand{\Mor}{\mathscr{L}}

\newcommand{\WF}{\operatorname{WF}}
\newcommand{\Id}{\operatorname{Id}}

\newcommand{\supp}[1]{{\mathrm{supp}(#1)}}
\newcommand{\pr}[1][]{\mathop{\mathrm{pr}_{#1}}\nolimits}

\newcommand{\Dil}{\mathop{\rho}\nolimits}
\newcommand{\Diff}{\mathop{\mathrm{Diff}}\nolimits}

\newcommand{\Op}{\mathop{\mathrm{Op}}\nolimits}
\newcommand{\op}{\mathop{\mathrm{op}}\nolimits}
\newcommand{\tr}{\mathop{\mathrm{tr}}\nolimits}

\newcommand{\orb}{\mathop{\mathrm{orb}}\nolimits}
\newcommand{\sg}{s_{{}_{\Gamma}}}
\newcommand{\rg}{r_{{}_{\Gamma}}}
\newcommand*{\dt}[1]{\accentset{\mbox{\large\bfseries .}}{#1}}
\newcommand{\mbt}{\raisebox{0.5mm}{\mbox{\large\bfseries .}}}
\newcommand{\dom}[1][]{\mathop{\mathrm{dom}_{#1}}\nolimits}

\begin{document}

\title[Evolution Equations on groupoids]
      {On evolution  equations for Lie groupoids }

 \author{Jean-Marie Lescure}
\address{Universit\'e Paris Est Créteil, CNRS, LAMA, F-94000 CRETEIL, FRANCE}

\email{jean-marie.lescure@u-pec.fr}

\author{St\'ephane Vassout}
\address{Université de Paris, CNRS, IMJ-PRG, F-75205 PARIS CEDEX 13, FRANCE}
\email{stephane.vassout@imj-prg.fr}

\date{\today}

\begin{abstract}
Using the calculus of Fourier integral operators on Lie groupoids developped in \cite{LV1}, we study the fundamental solution  of the evolution equation $ (\frac{\partial}{\partial t} +iP)u =0$ where $P$ is a self adjoint elliptic order one $G$-pseudodifferential operator on the Lie groupoid $G$. Along the way, we continue the study of distributions on Lie groupoids done in \cite{LMV1} by adding the reduced $C^{*}$-algebra of $G$ in the picture and  we investigate the local nature of the regularizing operators of \cite{Vassout2006}.
\end{abstract}

\maketitle

\setcounter{tocdepth}{2}
\tableofcontents  

\section{Introduction}
 
The main motivation of this paper is the construction of an approximate solution to the problem
\begin{equation}\label{intro:Cauchy-PB-evolution}
 \begin{cases}
     (\dfrac{\partial}{\partial t} + iP)u= f \\
       u(0)= g 
 \end{cases}
 \end{equation}
in the framework of a Lie groupoid $G\rightrightarrows M$. This means that $P$ here is a suitable order $1$ pseudodifferential $G$-operator, that $f,g$ live in suitable spaces of distributions and that the approximate solution will be seeked among Fourier integral $G$-operators. The present article can be considered as a continuation of  \cite{LMV1}, where properties of distributions on Lie groupoids, and convolution of them,  are studied in a certain generality, and of  \cite{LV1}, where H\"ormander's notion and calculus of Fourier integral operators on manifolds \cite{Horm-FIO1,Horm-classics} are exported to the framework of Lie groupoids.  We will frequently refer to the results of these papers, and  one of their cornerstones, namely the symplectic groupoid structure of $T^{*}G$ \cite{CDW}: 
$ \sg, \rg : \Gamma = T^{*}G \rightrightarrows A^{*}G$, will be of great importance here again.  

The Cauchy problem \eqref{intro:Cauchy-PB-evolution} has been and can be of course investigated in many situations under many different assumptions. We refer more precisely to \cite[Theorem 29.1.1]{Horm-classics} to illustrate the kind of results that we want to achieve on Lie groupoids. This can be summarized  by the following problem: 

\bigskip $ (\cP)$ \emph{Under an ellipticity assumption on $P$, the fundamental solution of  \eqref{intro:Cauchy-PB-evolution} should have, up to suitable regularizing error terms, an explicit approximation by Fourier integral $G$-operators that describes  in a simple and geometric way how the singularities of the initial data $g$ propagate at time $t$ under the action of the principal symbol of $P$. } 

\bigskip To set the problem on firm foundations, we first study in Section  \ref{sec:generalite-E-itP} existence and unicity conditions for  \eqref{intro:Cauchy-PB-evolution} in the general framework of $C^{*}$-algebras and Hilbertian modules, and we require there that $P$ is an unbounded self-adjoint regular operator on a  $C^{*}$-algebra $A$ \cite{BaajPhD,Baaj1983,WO,JenThom,Vassout2006,skand-cours2015}. Then the fundamental solution of  \eqref{intro:Cauchy-PB-evolution} denoted by $E(t)=e^{-itP}$ is obtained by continuous functional calculus, which yields the existence of solutions, while easy computations identical to those for Hilbert spaces show the uniqueness. We get in particular:
\begin{introthm}\label{intro:thm1} Let $A$ be a $C^{*}$-algebra, let $H$ be a Hilbertian $A$-module and $P$ be a selfadjoint regular operator on $A$. Let $H^{\infty}=\cap_{k}\dom P^{k}$. Then for any  $f\in C^\infty(\RR, H^\infty)$ and $g\in H^\infty $, the Cauchy problem  \eqref{intro:Cauchy-PB-evolution}
has a unique solution in $ C^{\infty}(\RR,H^{\infty})$, given by  
\begin{equation}\label{intro:sol-cauchy-complete}
  u(t) = e^{-itP}g + \int_0^t e^{i(s-t)P}f(s)ds.
\end{equation}
\end{introthm}

This preliminary result allows us to speak about the fundamental solution of  \eqref{intro:Cauchy-PB-evolution}  in the case  of a Lie groupoid $G$ with compact units space $M$ and of a first order elliptic symmetric and compactly supported, polyhomogeneous pseudodifferential $G$-operator $P$. Indeed, we then know by \cite{Vassout2006} that (the closure of) $P$ is  selfadjoint and regular on, for instance, the reduced $C^{*}$-algebra of $G$, denoted by $C^{*}_{r}(G)$. In particular, the theorem above applies and the task to find a nice approximation to $E(t)$ among Fourier integral $G$-operators is meaningful. Note that, because of \eqref{intro:sol-cauchy-complete}, the error term will automatically belong to the space  {$\cH^{\infty} = H^{\infty}\cap (H^{\infty})^{*}$}. Our answer to the problem  $(\cP)$  is the main result of the paper:
\begin{introthm}\label{intro:main-thm}
There exists a   $C^\infty$ family $\Lambda_t\subset T^*G$ of $G$-relations and a $C^\infty$ family of compactly supported Fourier integral $G$-operators $U(t)\in I^{(n^{(1)} - n^{(0)})/4}(G,\Lambda_t;\Omega^{1/2})$ such that :
\begin{equation}\label{intro:parametrix}
 (\frac{\partial}{\partial t} + iP) U(t) \in C^{\infty}_c(G, \Omega^{1/2}),
\end{equation}
and for any $t$, we have:  $ \ds E(t) - U(t) \in  {\cH^{\infty} }$.  
\end{introthm}
Let us now explain in some details the ingredients and the intermediate results, some of them being interesting on their owns,  required in the proof of the main theorem. 

First of all, Theorem \ref{intro:main-thm} immediately rises the preliminary question of the regularity of elements in  $ {\cH^{\infty} }$. Strictly speaking, elements of $ {\cH^{\infty} }$ live in a noncommutative $C^{*}$-algebra so dealing with their  local properties  makes a priori no sense. We manage on the one hand to prove that elements of the reduced $C^{*}$-algebra $C^{*}_{r}(G)$ of a Lie groupoid $G$ are distributions on $G$, in a canonical way, and on the other hand, to precise the regularity of elements in $ {\cH^{\infty} }$.  These intermediate tasks are the subject of Sections  \ref{sec:dist-test-facto} and \ref{sec:Sobolev} and the details can be summarized as follows.

The space of distributions we deal with, denoted by $\cD'(G)$, is  the one of  distributions on $G$ with values in the density bundle 
$ \Omega^{1/2} := \Omega^{1/2}(r^{*}AG)\otimes \Omega^{1/2}(s^{*}AG) $  
and thus the space of test functions we use, denoted by $\cD(G)$,  is the one of  compactly supported $C^{\infty}$ sections of the density bundle 
$ \Omega^{1/2}_{0} := \Omega^{-1/2} \otimes \Omega^{1}_{G}$. 
Thus $\cD'(G)= (\cD(G))'$ and the choice of $\Omega^{1/2}$ is relevant because $C^{\infty}_{c}(G) := C^{\infty}_{c}(G,\Omega^{1/2})\subset \cD'(G)$ is in a \emph{canonical} way an involutive algebra, whose a certain completion is precisely the algebra $C^{*}_{r}(G)$. Also, we have proved in \cite{LMV1} that the product $\star$ in $C^{\infty}_{c}(G)$ (called the convolution product for obvious reasons) widely generalizes, by continuity, to distributions in $\cD'(G)$. For instance the space $\cE'_{r,s}(G)$ of compactly supported distributions on $G$ whose pushforwards by the source and range maps are $C^{\infty}$ on $M$ (transversal distributions) forms a unital involutive algebra for the convolution product. Another justification for these choices of densities comes from the present work, indeed we prove that transversal distributions also act by convolution on $\cD(G)$ in a nice way, and that weak factorizations in the sense of \cite{Dix-Mal78} are available:

\begin{introthm}\label{intro:thm2} Let $\langle \cdot, \cdot \rangle$ denote the pairing $\cD'(G)\times \cD(G)\to \CC$ and $\iota$ the inversion of the groupoid. 
\begin{enumerate}
\item $ \ds \forall (u,\omega)\in \cD'_{r,s}(G)\times \cD(G), \quad \langle u, \omega \rangle =   \langle \iota^{*}u, \iota^{*}\omega \rangle =  \langle \delta_{M},  \iota^{*}u \star \omega \rangle$. 
\item The space $\cD(G)$ is a bimodule over $\cE'_{r,s}(G)$ and: 
\begin{align*}
 \forall u, v \in \cE'_{r,s}(G), \forall \omega\in \cD(G),\quad \langle u \star v, \omega \rangle  =  \langle  v, \iota^{*}u \star \omega \rangle =  \langle  u ,   \omega \star \iota^{*}v \rangle 
 \end{align*}
\item Let  $\omega \in \cD(G)$. For any  neighborhood $V$ of $M$ into $G$ ,  on can write $ \omega$ as a finite sum of elements 
 $\xi \star \chi$
where $\xi \in C^{\infty}_{c}(G)$,   $\supp{\xi}\subset V$ and $\chi\in \cD(G)$, $\supp{\chi}\subset \supp \omega$.
\end{enumerate}
\end{introthm}
This material allows us in Section  \ref{sec:Sobolev} to answer to the question about the local nature of elements in ${\cH^{\infty} }$, and along the way, that of elements in $C^{*}_{r}(G)$:

\begin{introthm}\ 
\begin{enumerate}
\item There is a continuous embedding $\ds C^{*}_{r}(G) \hookrightarrow \cD'(G)$.
 This embedding extends the  pairing $ \langle \cdot , \cdot \rangle$ between 
$ C^{\infty}_{c}(G)$ and $ \cD(G)$. 
\item The  inclusions  $\ds C^{\infty,0}_{\orb}(G)\cap C_c(G) \subset   {\cH^{\infty} } \subset  C^{\infty,0}_{\orb}(G)\cap C^{*}_{r}(G)$ hold true. \\
Here $ C^{\infty,0}_{\orb}$ refers to the space of continuous functions on $G$ that are $C^{\infty}$ on the subgroupoids $G_{\cO}=s^{-1}(\cO)$, as well as all their derivatives along the fibers of $s$ and $r$, for every orbits $\cO = r(s^{-1}(\{x\}))$ in $M$.
\end{enumerate}
\end{introthm}
Next, we explain how the principal symbol $p$ of $P$ gives rise to the  family of Lagrangian submanifolds $\Lambda_t$, $t\in\RR$, that will describe the propagation of singularities as expected in Problem $ (\cP)$. By definition, $P\in I^{ 1+(n^{(1)} - n^{(0)})/4} (G,M,\Omega^{1/2})$ is a polyhomogeneous conormal distribution, thus posseses a homogeneous principal symbol $p^{0}_{0}\in C^\infty(A^*G\setminus 0)$. If one considers the family $(P_x)_{x\in M}$ of ordinary pseudodifferential operators in the fibers of $s$ and collects their principal symbols into a homogeneous function  $p^{0}\in C^\infty(T_{s}^*G\setminus 0)$, $T^{*}_{s}G = (\ker ds)^{*}$, that will be called the \emph{principal $G$-symbol of} $P$. After lifting $p^{0}$ to a function on $T^*G\setminus \ker \rg$, one gets the following identity:
\begin{equation*}
  \forall (\gamma,\xi) \in T^*G\setminus \ker \rg, \quad p(\gamma,\xi) = p_0(\rg(\gamma,\xi)). 
\end{equation*}
Here $\rg$ is the range map of the symplectic groupoid $\Gamma = T^*G$. The computations also give a local expression for the sub-principal $G$-symbol of $P$, that is, for the collection of the sub-principal symbols of the operators $P_x$. Now it turns out that the Hamiltonian flow $\chi$ of the principal $G$-symbol $p$ is complete and right invariant, and we get the required Lagrangian submanifolds that will describe the evolution of singularities:
\begin{equation*}
\forall t\in\RR, \qquad \Lambda_{t} = \chi_{t}(A^{*}G\setminus 0).
\end{equation*}
This already produces a $C^\infty$ family of homogeneous Lagrangian submanifolds of $T^{*}G\setminus 0$ that satisfies the group relation, with respect to the product in $T^{*}G$:
$$ \Lambda_{t}.\Lambda_{s} =  \Lambda_{t+s} . $$
Moreover $\Lambda_{t}\subset \dt{T}^{*}G := T^{*}G\setminus(\ker \rg\cup\ker \sg)$, that is, in the vocabulary of \cite{LV1}, every $\Lambda_{t}$ is a $G$-relation, while the global object coming with   the family $(\Lambda_{t})_{t}$:
$$ \Lambda = \{ (t,\tau,\gamma,\xi) \ ; \ \tau + p(\gamma,\xi) = 0, \ (\gamma,\xi)\in \Lambda_{t} \} \subset T^{*}\RR\times T^{*}G $$
is a family $\RR\times G$-relation. As in \cite{LV1}, this construction highlights the important role of the symplectic groupoid structure of $T^{*}G$ in analysis.

There is a last result, of technical nature, that intervenes in the proof of Theorem \ref{intro:main-thm}. Indeed, assuming that the Lagrangian submanifolds $\Lambda_{t}$ provide the good candidate for Theorem \ref{intro:main-thm}, we are led to search a first order parametrix $U_{0}$ for $\partial_{t} + i P$  among  the Fourier integral $G$-operators associated with $(\Lambda_{t})_{t}$. This amounts to solve the transport equation for principal symbols: 
$$ \frac{\partial}{\partial t}\sigma_{\pr}(U_{0}) + i \sigma_{\pr}(PU_{0}) = 0$$ 
and thus it requires to express the principal symbol of the convolution product $PU_{0}$ of the lagrangian distributions $P$ and $U_{0}$. Since by construction and on purpose, the principal symbol $p_{0}$ vanishes on $\rg\Lambda$, we need to look for the next term in the asymptotic expansion of the total symbol of $PU_{0}$. This is what is achieved, modulo some technical details, by using the following result: 
\begin{introthm}
Let $Q\in\Psi_c^m(G)$, with principal $G$-symbol $q$, sub-principal symbol $q^{1s}$, and let $C$ be a $G$-relation such that  $q$ vanishes on $C$. 
Let $A\in I^{m'}(G,C;\Omega^{1/2})$ and $a$ be a principal symbol of $A$. \\
Then 
\begin{equation*}
   QA \in I^{m+m'-1}(G,C;\Omega^{1/2}) \text{ and }   \sigma_{\pr}(QA) = -i\cL_{q}a+ q^{1s}a.
\end{equation*}
Here $\cL_{q}$ is the Lie derivative along the Hamiltonian vector field $H_{q}$ of $q$.
\end{introthm}

Many interesting situations produce non compactly supported operators $P$: for instance, if $\Delta$ is a Laplacian on $G$ then $\sqrt{\Delta} = P + S$ with $P$ as above and $S\in \cH^{\infty}$ \cite{Vassout2006}. The main theorem trivially extends to such non compactly supported operators: one just needs to replace $C^{\infty}_{c}(G)$ by ${\cH^{\infty} }$ in \eqref{intro:parametrix}. We describe at the end of the paper several situations where Theorem \ref{intro:main-thm} applies:
\begin{enumerate}
\item The usual  pseudodifferential calculus on a compact manifold without boundary $X$. We use the pair groupoid $ G = X  \times X \rightrightarrows X$. Since $X$ itself is an orbit, we have $C^{\infty,0}_{\orb}(G,E)=C^{\infty}(X\times X,E)$ and we just recover the classical result on manifolds. 
\item The longitudinal calculus on foliations \cite{Connes1979}. We use  the holonomy groupoid. We recover a construction in  \cite{Kordyukov1994} by  a quite different approach.
\item Right invariant calculus on a Lie group $G$. We use $G$ as a groupoid with units space $\{ e \}$. 
\item The $b$-calculus on manifolds with corners \cite{Melrose-Piazza1992}. We use the $b$-groupoid  \cite{Monthubert2003}. 
\item The calculus on manifolds with fibred boundary or with iterated fibred corners  \cite{Mazzeo-Melrose,DLR}. We use  the groupoid of  \cite{DLR}.  
\end{enumerate}
As far as we know, the results obtained for cases (3), (4), (5) above are new. 

The next section contains the basic definitions and notation necessary for the sequel and can be considered as an extension of the introduction for the unfamiliar reader. 

\bigskip \noindent \textbf{Acknowledgments}
The authors are grateful to Claire Debord, Omar Mohsen, Victor Nistor and Georges Skandalis for helpful and stimulating  discussions or remarks. The first author is thankful for the hospitality of  the IMJ-PRG, Paris University, where part of this project was achieved. Most part of this work has been realized with the support of the  Grant ANR-14-CE25-0012-01  SINGSTAR.

\section{Notation and reminders}\label{sec:generalities}

\noindent\textbf{Densities on manifolds.} If $E$ is a real vector space of dimension $n$ and $\alpha\in\RR$, we denote by $\Omega^{\alpha}(E)$ the vector space of maps $\omega : \Lambda^{n}E\setminus  0\to \CC$, called $\alpha$-densities, such that $\omega(t V) = \vert t\vert^{\alpha}\omega(V)$ for any $t \not= 0$ and $V\in \Lambda^{n}E\setminus 0$. 
For any $C^{\infty}$ real vector bundle $E\to X$, the vector bundle $\Omega^{\alpha}(E) = \cup_{x}\Omega^{\alpha}(E_{x})\to X$ is a $C^{\infty}$ line bundle, with transition functions given by $\vert \det( g_{ij})\vert^{\alpha}$ if $(g_{ij})$ is a set of transition functions for $E$. Sections of $\Omega^{\alpha}(E)$ are called $\alpha$-densities on $E$ and sections of $\Omega^{\alpha}_{X} := \Omega^{\alpha}(TX)$ are called $\alpha$-densities on $X$.  Densities bundles are always trivialisable, but not canonically in general: one can construct an everywhere positive section using local trivializations.  

A fundamental point is that compactly supported one densities on $X$ can be integrated over $X$. More precisely, there is a  unique linear form $ \int_{X} : C^{\infty}_{c}(X,\Omega^{1}_{X})\To \RR$ such that if $f = \mathsf{f}(x)\vert dx\vert$ is compactly supported in a local chart $U$ with local coordinates $x=(x_{1},\ldots,x_{n})$, then 
$$ \int_{X} f = \int_{\RR^{n}} \mathsf{f}(x) dx.$$
Above, $\vert dx\vert$ is the one density defined by $ \vert dx\vert =\vert dx_{1}\wedge \cdots \wedge dx_{n}\vert$.  Diffeomorphisms $\phi : X \to Y$ provide isomorphisms $\phi^{*} : \Omega^{\alpha}_{Y} \to \Omega^{\alpha}_{X}$ given by $\phi^{*} \omega(V) = \omega( \phi_{*}V)$. By construction, the integral of one densities is invariant under the action of diffeomorphisms. Densities are usually handled with the following canonical isomorphisms: 
\begin{enumerate}[-]
\item $\Omega^{\alpha}(E)\otimes \Omega^{\beta}(E) \simeq \Omega^{\alpha+\beta}(E)$
\item $\Omega^{\alpha}(E\oplus F) \simeq \Omega^{\alpha}(E)\otimes \Omega^{\alpha}(F)$,
\item $\Omega^{\alpha}(E^{*}) \simeq \Omega^{-\alpha}(E)$
\item if $0\to F \to E \to G \to 0$ is exact, then $\Omega^{\alpha}(E)\simeq \Omega^{\alpha}(F)\otimes\Omega^{\alpha}(G)$. 
\end{enumerate}

\noindent\textbf{Lie groupoids.} A Lie groupoid $G \rightrightarrows M$ is a pair of manifolds $(G,M)$ of respective dimensions generally denoted by $n=n^{(1)}+n^{(0)}$ and $n^{(0)}$, together with the following data and required properties.  The data are:
\begin{enumerate}[(a)]
\item two surjective submersions $r,s :G\to M$, called range and source,
\item a $C^{\infty}$ section $\upsilon : M \to G$ of both $r$ and $s$, assimilated to an inclusion,
\item a $C^{\infty}$  map $\iota : G \to G$ called inversion, noted: $\gamma^{-1} := \iota(\gamma)$,
\item a $C^{\infty}$  map $G^{(2)}=\{ (\gamma_{1},\gamma_{2})\in G^{2} \ ; \ s(\gamma_{1})=r(\gamma_{2})\}  \to G$ called multiplication: $\gamma_{1}\gamma_{2} :=m(\gamma_{1},\gamma_{2})$.
\end{enumerate}
The required properties are those giving a sense to the following intuition: a groupoid is the algebraic structure obtained from a group $G$ after spreading out its unit into a whole subset $M$, that is
\begin{enumerate}[(i)]
\item $r(\gamma_{1}\gamma_{2}) = r(\gamma_{1})$, $s(\gamma_{1}\gamma_{2}) = s(\gamma_{2})$ whenever it makes sense,
\item $\upsilon(r(\gamma))\gamma = \gamma$, $ \gamma \upsilon(s(\gamma)) = \gamma$ for all $ \gamma$,
\item $r(\gamma^{-1})= s(\gamma)$, $s(\gamma^{-1}) = r(\gamma)$  for all $ \gamma$,
\item $ \gamma \gamma^{-1} = \upsilon(r(\gamma))$, $ \gamma^{-1} \gamma = \upsilon(s(\gamma))$  for all $ \gamma$,
\item  $ (\gamma_{1}\gamma_{2})\gamma_{3} =  \gamma_{1}(\gamma_{2}\gamma_{3})$ whenever it makes sense.
\end{enumerate}
It follows that  $\upsilon$ is an embedding (often omitted in the notation), that $\iota$ is an involutive diffeomorphism and $m$ a surjective submersion. 
We note $G_{x}$, the $s$-fiber at $x\in M$, $G^{x}$ its $r$-fiber, and we set $T_{s}G = \ker ds$, $T_{r}G = \ker dr$. We note $L_{\gamma}$, $R_{ \gamma}$ the left and right multiplication by $ \gamma$. 
The Lie algebroid $AG$ of $G \rightrightarrows M$ is by definition here the vector bundle $\ker ds \vert_{M} \to M$. The differential map $ dr : AG \to TM$ is denoted by $\mathfrak{a}$ and called the anchor map. To any $C^{\infty}$ section $X\in \Gamma(AG)$ corresponds a right invariant vector field $\widetilde{X}\in \Gamma(TG)$, defined by $\widetilde{X}_{\gamma} := dR_{\gamma}(X_{r(\gamma)})$, and conversely. The right invariance means  $\widetilde{X}_{\gamma\eta}= dR_{\eta}(\widetilde{X}_{\gamma})$. This allows to define a Lie algebra structure on $\Gamma(AG)$ that satisfies 
$$\forall X,Y\in \Gamma(AG),\  \forall f\in C^{\infty}(M), \quad  [ X, fY ] = f[ X,Y] + (\mathfrak{a}(X)f) Y .$$
We refer to  \cite{Moerdijk2003,Mackenzie2005} for a detailed account on Lie groupoids and Lie algebroids. 

We will use several $\alpha$-densities bundles over $G$, often for $\alpha = \pm 1/2, \pm 1$: 
\begin{enumerate}[-]
\item The bundles $\Omega^{\alpha}(\ker d\pi)$ of densities along the fibers of $\pi = s,r$. They are conveniently replaced for computations by the respective isomorphic bundles $\Omega^{\alpha}_{s} = \Omega^{\alpha}(r^{*}AG)$ and $ \Omega^{\alpha}_{r} = \Omega^{\alpha}(s^{*}AG)$. The isomorphisms are induced  by: 
 \begin{equation}\label{nota:cR}
  \cR : r^{*}AG \To T_sG,\quad  (\gamma, X) \longmapsto \big(\gamma, (dR_{\gamma})_{r(\gamma)}(X)\big), 
 \end{equation}
  \begin{equation}\label{nota:cS}
  \cS : s^{*}AG \To T_rG, \quad  (\gamma, X) \longmapsto \big(\gamma, (dL_{\gamma}\circ \iota)_{s(\gamma )}( X)\big).
 \end{equation}
 \item  The ``symmetrisation'' of the preceeding ones: $\Omega^{\alpha} = \Omega^{\alpha}_{s} \otimes \Omega^{\alpha}_{r}$, which is suitable for convolution on $G$. 
\item The bundle $\Omega^{1/2}_{0} = \Omega^{-1/2} \otimes \Omega^{1}_{G}$ necessary for the pairing: 
\begin{equation*}
f\in C^{\infty}(G,\Omega^{1/2}), \ \omega\in C^{\infty}_{c}(G,\Omega^{1/2}_{0}), \quad \langle f , \omega \rangle = \int_{G} f g . 
\end{equation*}
Actually, there is a natural isomorphism $\Omega^{1/2}_{0}\simeq \Omega^{1/2}(r^{*}TM)\otimes \Omega^{1/2}(s^{*}TM)$.
\end{enumerate}
\medskip \noindent\textbf{The cotangent groupoid}
The cotangent space $T^{*}G$ has a non trivial groupoid structure: $\Gamma = T^{*}G \rightrightarrows A^{*}G$, with structure maps 
$\rg,\sg,m_{{}_{\Gamma}},\iota_{{}_{\Gamma}}$ defined as follows:
\begin{enumerate}[-]
\item $\rg(\gamma,\xi) = \big(r(\gamma), {}^{t}dR_{\gamma}(\xi\vert_{T_{\gamma}G_{s(\gamma)}})\big)$ and  $\sg(\gamma,\xi) = \big(s(\gamma), -{}^{t}d(L_{\gamma}\circ \iota)(\xi\vert_{T_{\gamma}G^{r(\gamma)}})\big)$,
\item $(\gamma_{1},\xi_{1})(\gamma_{2},\xi_{2}) = (\gamma_{1}\gamma_{2}, \xi)$ with $\xi(dm(t_{1},t_{2})) = \xi_{1}(t_{1}) +  \xi_{2}(t_{2})$,
\item $(\gamma,\xi)^{-1} = (\gamma^{-1}, - {}^{t}d\iota(\xi))$.
\end{enumerate}
This is a \emph{symplectic} groupoid, which means that the graph of $m_{{}_{\Gamma}}$ is a Lagrangian submanifold of $(T^{*}G)^{3}$ provided with the symplectic form $\omega\oplus \omega\oplus-\omega$, with $\omega$ the canonical symplectic form of $T^{*}G$. We refer to \cite{CDW,Mackenzie2005} for a detailed account on symplectic groupoids and on the related notion of $VB$-groupoids, as well as to \cite{LMV1,LV1} for the interest of this symplectic structure regarding the theory of distributions on groupoids. We will denote $$\quad  T^{*}_{\mbt} G = T^{*}G\setminus \ker \rg \ \text{ and } \quad \dt{T}^{*}G = T^{*}G\setminus (\ker \rg\cup\ker \sg). $$

We will consider in this paper homogeneous lagrangian submanifods of $T^{*}G\setminus 0$ that avoids the kernel of $\rg$ and $\sg$. We call them $G$-relations, in reference to the term \emph{canonical relations} often employed for (product) manifolds. Under mild assumptions, $G$-relations compose well in the groupoid $T^{*}G$ \cite{LV1}. $G$-relations $\Lambda$ such that $\sg\vert_{\Lambda}$ and $\rg\vert_{\Lambda}$ are diffeomorphisms onto their ranges are called invertible. 
 We will sometimes use densities along the $\sg$ and $\rg$-fibers of the cotangent groupoid $T^{*}G$.  Both are naturally isomorphic and: 
 \begin{equation*}
\Omega^{ \alpha}_{\sg} \simeq \Omega^{ \alpha}_{\rg} \simeq  \hat{\Omega}^{ \alpha} \otimes \hat{\Omega}^{-\alpha}_{G}
\end{equation*}
where $\hat{E}$ denotes the pull back to $T^{*}G$ of the bundle $E\to G$. Also, we note that $\Omega^{1}_{\sg} \vert_{A^{*}G} = \Omega^{1}(AT^{*}G) = (\cD_{AG}^{\tr})^{-1}$ where $\cD_{AG}^{\tr}$ is the transverse density bundle of $AG$ \cite{Crainic-Mestre2019}.

\medskip\noindent\textbf{The convolution algebra} Throughout this paper we make the convention:
\begin{equation*}
C^{\infty}(G) : = C^{\infty}(G,\Omega^{1/2}), 
\end{equation*}
that is, we omit the ubiquitous density bundle $\Omega^{1/2}$ in the notation. We apply the same convention for other sections of $\Omega^{1/2}$ with various regularity and support conditions. When the sections of a different bundle are considered, this bundle will be always mentionned. 

The \emph{convolution algebra structure} on $C^{\infty}_{c}(G)$ refers to the product $\star$ canonically defined from any of  the following three intuitive formulas:
\begin{equation}\label{nota:intuitive-star}
f \star g (\gamma) =  \int_{\gamma_{2}\in G_{s(\gamma)}} f(\gamma\gamma_{2}^{-1})g(\gamma_{2}) =  \int_{\gamma_{1}\in G^{r(\gamma)}} f(\gamma_{1})g(\gamma_{1}^{-1}\gamma) = \int_{(\gamma_{1},\gamma_{2})\in m^{-1}(\gamma)} f(\gamma_{1})g(\gamma_{2})
\end{equation}
This is justified as follows. Write $f = \mathsf{f} (\mu_{s}\mu_{r})^{1/2}$,  $g = \mathsf{g} (\mu_{s}\mu_{r})^{1/2}$ with $ \mathsf{f},\mathsf{g}\in C^{\infty}_{c}(G,\CC)$, $\mu_{s}=r^{*}\mu\in C^{\infty}(G,\Omega^{1}_{s})$, $\mu_{r}=s^{*}\mu \in C^{\infty}(G,\Omega^{1}_{r})$ for some positive $\mu\in C^{\infty}(M,\Omega^{1}(AG))$. Then, whenever $\gamma_{1}\gamma_{2} = \gamma$:
\begin{equation*}
 f(\gamma_{1})g(\gamma_{2}) =  \mathsf{f}(\gamma_{1})\mathsf{g}(\gamma_{2}) \mu_{r}^{1/2}(\gamma_{1})\mu_{s}^{1/2}(\gamma_{2}) (\mu_{s} \mu_{r})^{1/2}(\gamma) \text{ and } \mu_{r}(\gamma_{1})=\mu_{s}(\gamma_{2}).
\end{equation*}
 We now may set rigorously: 
 \begin{equation}\label{nota:rigorous-star-s-style}
 f \star g (\gamma) = \Big( \int_{\gamma_{2}\in G_{s(\gamma)}}  \mathsf{f}(\gamma\gamma_{2}^{-1})\mathsf{g}(\gamma_{2}) \cR_{*}\mu_{s}(\gamma_{2})\Big) (\mu_{s} \mu_{r})^{1/2}(\gamma).
 \end{equation}
 This gives consistance to the first formula in \eqref{nota:intuitive-star}. The second and third formulas are obtained from the first one using the diffeormorphisms  $ L_{\gamma}\circ \iota : G_{s(\gamma)} \to G^{r(\gamma)}$ and $ (L_{\gamma}\circ \iota, \Id) : G_{s(\gamma)} \to m^{-1}(\gamma)$. Equivalently, one  can directly define them as we did for the first one using the suitable structural isomorphisms to create the appropriate one densities on $G^{r(\gamma)}$ and $m^{-1}(\gamma)$. With the notation above, this concretely means: 
\begin{align}\label{nota:rigorous-star-r-m-style}
 f \star g (\gamma) & = \Big( \int_{\gamma_{1}\in G^{r(\gamma)}}  \mathsf{f}(\gamma_{1})\mathsf{g}(\gamma_{1}^{-1}\gamma) \cS_{*}\mu_{r}(\gamma_{1})\Big) (\mu_{s} \mu_{r})^{1/2}(\gamma) \\
  & = \Big( \int_{(\gamma_{1},\gamma_{2})\in m^{-1}(\gamma)}  \mathsf{f}(\gamma_{1})\mathsf{g}(\gamma_{2}) \cM_{*}\mu_{s}(\gamma_{2})\Big) (\mu_{s} \mu_{r})^{1/2}(\gamma).
 \end{align}
 with, for the last line:
 \begin{equation}\label{nota:cM}
  \cM : r^{*}AG\vert_{G_{s(\gamma)}} \To T m^{-1}(\gamma), \quad (\gamma_{2}, X) \longmapsto \big(\gamma\gamma_{2}^{-1}, \gamma_{2}, d(L_{\gamma\gamma_{2}^{-1}}\circ \iota)( X), dR_{\gamma_{2}}(X)\big) .
 \end{equation}

By $C^{\infty,0}_{\pi}(G)$, we denote the space  of elements in $C_{c}(G)$ that  belong to $C(U_{(0)},C^{\infty}(U_{(1)}))$ over any local trivializations $\kappa : U \overset{\simeq}{\to} U_{(0)}\times U_{(1)}$ of $\pi$ (here $\pi = \pr_{1}\circ\kappa$). The topology is modeled on that of  $C(U_{(0)},C^{\infty}(U_{(1)}))$ and is Fréchet. We write $C^{\infty,0}_{\pi,c}$ for $C^{\infty,0}_{\pi}\cap C_{c}$, and equip it with the corresponding LF-topology.

\medskip\noindent\textbf{The reduced $C^{*}$-algebra of a groupoid.}
The space  $C^\infty_c(G,\Omega^{1/2}_s)$ comes with a natural prehilbertian $C(M)$-module structure: 
\begin{equation}
    f,g\in C^\infty_c(G,\Omega^{1/2}_{s}),\quad \langle f \ \vert \ g \rangle_{s}(x) = \int_{G_{x}}\overline{f(\gamma)}g(\gamma). 
\end{equation}
Its completion as a hilbertian $C(M)$-module is denoted by $ L_{s}^{2}(G)$.
The homomorphism $\lambda : C^\infty_c(G)\longrightarrow \cL(L_{s}^{2}(G))$ given by: 
$$ \forall f \in C^\infty_c(G), \ g \in C^\infty_c(G,\Omega^{1/2}_{s}), \quad \lambda(f)(g)(\gamma)  = f \star g(\gamma) = \int_{G_{s(\gamma)}}f(\gamma\alpha^{-1})g(\alpha) $$
is well defined, injective,  and the reduced $C^{*}$-algebra of $G$, denoted by $C^*_{r}(G)$,  is the completion of $C^\infty_c(G)$ with respect to the $C^{*}$-norm $\| f\| = \| \lambda(f)\|_{\mathrm{op}}$. The extended homomorphism 
$$\lambda : C^*_{r}(G)\longrightarrow \cL(L_{s}^{2}(G))$$ 
is called the left regular representation.  Starting from $C^\infty_c(G,\Omega^{1/2}_r)$, we get a Hilbert $C(M)$-module $L^{2}_{r}(G)$ and the right regular representation $\rho : C^*_{r}(G)\longrightarrow \cL(L_{r}^{2}(G))$. The adjunction map $* : L_{s}^{2}(G)\To L_{r}^{2}(G)$ provides a unitary anti-homomorphism. The unfamiliar reader  may consult \cite{Renault1980,Connes1994,KhosSkand2} for groupoids $C^{*}$-algebras and  \cite{WO,JenThom,skand-cours2015} for Hilbertian modules.

\medskip\noindent\textbf{Distributions.}
We consider in this article various spaces of distributions on $G$, always valued in $\Omega^{1/2}$, which thus is safely omitted.  We set:
\begin{equation*}
\cD'(G) := \cD'(G,\Omega^{1/2}). 
\end{equation*}
This is the topological dual of the space: 
\begin{equation*}
\cD(G) := C^{\infty}_{c}(G, \Omega^{1/2}_{0}),
\end{equation*}
where $\Omega^{1/2}_{0} := \Omega^{-1/2}\otimes\Omega^{1}_{G}$.  The elements of $ \cD(G)$ will be called test functions, with a slight abuse of vocabulary. We denote by $\cE'(G)$ the subspace of $\cD'(G)$ consisting of compactly supported distributions. We set: 
\begin{equation}
\cD'_{\pi}(G) = \{ u \in \cD'(G) \ ; \ \forall f\in C^{\infty}_{c}(G,\Omega^{1/2}_{s}),\ \pi_{!}(uf) \in C^{\infty}(M,\Omega^{1/2}(AG)) \}
\end{equation}
where $\pi_{!}$ denotes the pushforward of distributions and $\pi = r,s$. Elements of $\cD'_{\pi}(G)$ are called  $C^{\infty}$-transversal distributions with respect to $\pi$ \cite{Androulidakis-Skandalis,LMV1,vEY}. The convolution product $\star$ extends by continuity to transversal distibutions, providing  $\cE_{\pi}'(G)$ with the structure of a unital algebra and  $\cE_{r,s}'(G) = \cE_{s}'(G)\cap \cE_{r}'(G)$ with the structure of an involutive unital algebra. The unit is $\delta_{M}(f)=\int_{M}f$ and the involution is $u^{*} = \overline{\iota^{*}u}$. Elements of $\cD'_{\pi}(G)$  can be restricted fiberwise, giving $C^{\infty}$ families over $M$ of distributions in the fibers, the space of whose families being denoted by $C^{\infty}_{\pi}(M,\cD'(G))$, or viewed as $C^{\infty}(M)$-linear continuous operators, the space of whose operators being denoted by $\cL_{C^{\infty}(M)}(C^{\infty}_{c}(G),C^{\infty}(M,\Omega^{1/2}(AG)))$,  and there are canonical isomorphisms:
\begin{equation*}
 \cD'_{\pi}(G) \simeq C^{\infty}_{\pi}(M,\cD'(G)) \simeq \cL_{C^{\infty}(M)}(C^{\infty}_{c}(G),C^{\infty}(M,\Omega^{1/2}(AG))).
\end{equation*}
We will also consider continuously transversal distributions with respect to $\pi = r,s$:
\begin{equation}
\cD'_{\pi,0}(G) = \{ u \in \cD'(G) \ ; \ \forall f\in C^{\infty}_{c}(G,\Omega^{1/2}_{s}),\ \pi_{!}(uf) \in C(M,\Omega^{1/2}(AG)) \}. 
\end{equation}
By rephrazing the arguments in \cite{LMV1}, one gets:
\begin{equation*}
 \cD'_{\pi,0}(G) \simeq C_{\pi}(M,\cD'(G)) \simeq \cL_{C(M)}(C^{\infty,0}_{\pi,c}(G),C(M,\Omega^{1/2}(AG))).
\end{equation*}

\medskip\noindent\textbf{$G$-operators: } they are the continuous linear maps $C^\infty_c(G)\to C^\infty(G)$ given by right invariant families of (linear continuous) operators acting in the $s$-fibers. More precisely, $P$ is a $G$-operator if there exists a family $P_x : C^\infty_c(G_x,\Omega^{1/2}_{G_x})\longrightarrow C^\infty(G_x,\Omega^{1/2}_{G_x})$, $x\in M$,   such that for all $x\in M$, $\gamma\in G$, $f \in C^\infty_c(G)$:
\begin{equation}\label{eq:axiom-G-op}
  P(f)|_{G_x} = P_x(f|_{G_x}) \text{ and } P_{r(\gamma)}\circ R_\gamma=R_\gamma \circ P_{s(\gamma)},\ \forall \gamma\in G.
\end{equation}
This is  equivalent to requiring that $P$ maps  $C^\infty_c(G)\to C^\infty(G)$ continuously and that:
$$ P(f) \star g = P(f \star g) \quad \text{ for any } f,g \in C^\infty_c(G). $$
A $G$-operator $P$ as an adjoint if there exists a  $G$-operator $Q$ such that $ (Pf)^{*}\star g = f^{*}\star (Qg)$ for any $f,g$. We denote by $\Op_{G}$ (resp. $\Op^{*}_{G}$, $\Op^{*}_{G,c}$) the space of (resp. adjointable, compactly supported and adjointable) $G$-operators. 

It is proved in \cite{LMV1} that the map 
\begin{equation*}
\cD'_{r}(G)\to \Op_{G},\quad u \mapsto u\star \cdot
\end{equation*}
is an isomorphism, with inverse $P \mapsto k_{P}$,  given by  
$k_{P}(\gamma) = p_{r(\gamma)}(r(\gamma),\gamma^{-1})$ where $p_{x}$ denotes the Schwartz kernel of $P_{x}$. The same map induces an isomorphism:
\begin{equation*}
\mathrm{Op}_{G,c}^{*}\simeq\cE'_{r,s}(G).
\end{equation*}

\medskip\noindent\textbf{Pseudodifferential $G$-operators and regularizing operators.} Among the class of $G$-operators one finds the well known subclass of pseudodifferential $G$-operators ($G$-PDO) \cite{Connes1979,MP,NWX,Vassout2006}, that is,  of right invariant families of pseudodifferential operators in the $s$-fibers: they coincide with left convolution by distributions in: 
\begin{equation}
 \Psi^*_{G} = I^{*+(n^{(1)} - n^{(0)})/4}(G,M;\Omega^{1/2}).
\end{equation}
where here $I$ refers to the space of conormal distributions. One has a principal symbol map:
\begin{equation*}
   \sigma_0: \Psi^m_{G}  \longrightarrow S^{[m]}(A^*G)
\end{equation*}
with kernel $ \Psi^{m-1}_{G}$. Here $S^{[m]} = S^{m}/S^{m-1}$. It is well known that  $(\Psi^*_{G,c},\star )$ is an involutive unital algebra and $\sigma_{0}$ an algebra homomorphism. 
When $P\in \Psi^{1}_{G,c}$ is elliptic and symmetric, then its closure, as an unbounded operator on $C^{*}_{r}(G)$ with domain $C^{\infty}_{c}(G)$, is selfadjoint and \emph{regular} \cite{Vassout2006,BaajPhD,Baaj1983,skand-cours2015}. There is a canonical scale $H^{t}$, $t\in\RR$, of Hilbert $C^{*}_{r}(G)$-modules, that we call intrinsic Sobolev modules, which do not depend, up to isomorphism of Hilbertian structures, on the symmetric elliptic operator $P\in \Psi^{1}_{G,c}$  used to define them: 
\begin{equation*}
 H^{t} = \overline{C^{\infty}_{c}(G)}^{ \langle \cdot | \cdot \rangle_{t} }, \quad  \langle f \ | \ g \rangle_{t}  =   \langle (1 + P^{2})^{t} \star f \ | \ g \rangle \in C^{*}_{r}(G),
\end{equation*}
where $ \langle a | b \rangle = a^{*}b$. Then any $Q\in \Psi^{m}_{G,c}$ gives a bounded homomorphism  $Q \in \cL(H^{t}, H^{t-m})$ and for any $ m > 0$, the inclusion $H^t\hookrightarrow H^{t+m}$ is a compact homomorphism of Hilbert modules. All of this material is developped in  \cite{Vassout2006}. Although we call the spaces $H^{t}$  \emph{Sobolev} modules, we may think of them as modules of abstract pseudodifferential operators of order $ < - t $. Indeed, $H^{t}$ is also the completion of $\Psi^{ < - t }_{G,c}$ for the norm $\| Q \|_{t} = \| (1 + P^{2})^{t/2} Q\|_{C^{*}_{r}(G)}$. Scales of Hilbert modules closer to the usual notion of Sobolev regularity of order $t$ for functions or distributions will be obtained using the left regular representation of $H^{t}$. 

The algebra $ \Psi^{*}_{G,c}$ is too small for practical purposes. For instance, the inverse of an elliptic element in $ \Psi^{*}_{G,c}$ which is invertible as an operator between Sobolev modules, has no reason to be compactly supported. This phenomenom propagates to operators obtained by holomorphic functional calculus and we will eventually face it also when building an approximation of $E(t)=e^{itP}$ by Fourier integral $G$-operators. A suitable enlargement of $\Psi^{*}_{G,c}$ is provided by: 
\begin{equation}\label{eq:enlarged-PDO}
\Psi^{*}_{G} : =\Psi^{*}_{G,c} +  \cH^{\infty}, \quad \text{ where } \cH^{\infty} = H^{\infty} \cap (H^{\infty})^{*}
\end{equation}
and $H^{\infty} = \cap_{t}H^{t} \subset C^{*}_{r}(G)$. Actually, $ \cH^{\infty} $ coincides with the ideal of regularizing operators introduced in \cite{Vassout2006}.

\medskip\noindent\textbf{Fourier integral $G$-operators.}
Another remarkable subclass of $G$-operators is given by that of lagrangian distributions on $G$ with respect to arbitrary $G$-relations. We call them Fourier integral $G$-operators ($G$-FIO) and we set for a given $G$-relation $\Lambda$:
 \begin{equation}
  \Phi^*_G(\Lambda)  = I^{*+(n^{(1)} - n^{(0)})/4}(G,\Lambda;\Omega^{1/2}).
 \end{equation}
 where here $I$ refers to the space of Lagrangian distributions.
The convolution product gives a map
\begin{equation}
  \Phi^*_{G,c}(\Lambda_1)\times \Phi^*_{G}(\Lambda_2)\to \Phi^*_{G}(\Lambda_1\Lambda_2)
\end{equation}
as soon as $\Lambda_1 \times \Lambda_2$ has a clean intersection with $T^*G^{(2)}$. This proves in particular that $\Phi^*_G(\Lambda)$ is a bimodule over 
 $\Psi^*_{G,c}$ and that $F^{-1}PF\in \Psi^*_{G,c}$ if $P\in \Psi^*_{G,c}$ and $F\in \Phi^*_{G,c}(\Lambda)$ is invertible. Also, when $\Lambda$ is invertible, one gets $ \Phi^{0}_{G,c}(\Lambda) \subset \cM(C^{*}_{r}(G))$ and $ \Phi^{ < 0}_{G,c}(\Lambda) \subset C^{*}_{r}(G)$. In general, if $A\in  \Phi^*_{G}(\Lambda)$, the corresponding family   $(A_{x})_{x\in M}$ consists of operators $A_{x}$ given by locally finite sums of oscillatory integrals and when $\Lambda$ is transversal to $T_{L}^{*}G$, for any $L = s^{-1}(\cO)$ and $\cO \in M/G$, (this is for instance the case if $\Lambda$ is invertible), then each $A_{x}$ is a genuine Fourier integral operator on the manifold $G_{x}$. All the statements here about $G$-FIOs are proved in \cite{LV1}. 

\section{The one parameter group $e^{-itP},\ t\in \RR$} \label{sec:generalite-E-itP}

Before analyzing evolution equations on groupoids, we study the functional analytic aspects of them in a reasonably general and simple framework. So, let us consider the Cauchy problem:
\begin{equation}\label{eq:cauchy-pb}
\begin{cases}
(\frac{\partial}{\partial t} +iP) u = f \\
 u(0) = g 
\end{cases}
\end{equation} 
in the following situation: $P$ is a regular self-adjoint operator on $H$ \cite{BaajPhD,Baaj1983,skand-cours2015,Vassout2006} where  $H$ is a Hilbert module over some $C^*$-algebra $A$. It turns out that under natural assumptions on $f$ and $g$, this problem has a unique solution given in term of the operator $e^{-itP}$. This operator is first defined in term of the unbounded continuous functional calculus for regular operators \cite[Paragraph 14.3.3]{skand-cours2015}.  We recall that any  nondegenerate representation
 \[ 
   \pi : C_0(\RR) \longrightarrow\Mor(H)
 \]
extends into a map $\widetilde{\pi}$ from $C(\RR)$ (viewed as regular operators on $ C_0(\RR)$)  to the set of regular operators on $H$. The map   $\widetilde{\pi}$ is defined through the identification $ C_0(\RR) \otimes_{\pi} H\simeq H$ and the formula: 
$$ \widetilde{\pi}(f) = f\otimes_{\pi}  \Id. $$
Moreover, there exists a unique such representation $\pi$ such that  $\widetilde{\pi}(\Id_{\RR}) = P$ and we fix this particular one from now on.   Introducing  $f_{t}\in C(\RR)$, $f_t(\lambda)=e^{-it\lambda}$, we set: 
\[
 e^{-itP} = \widetilde{\pi}(f_t).
\]
Actually, the restriction of $  \widetilde{\pi}$  to $C_{b}(\RR)$ is a strictly continuous homomorphism \cite[Proposition 5.19]{skand-cours2015} : 
$$ \overline{\pi} = \widetilde{\pi}|_{C_{b}(\RR)} : C_b(\RR) \longrightarrow \cL(H). $$
Here, strict continuity refers to the topologies of $C_b(\RR)$ and $ \Mor(H)$ as multiplier algebras of $C_0(\RR)$ and $\cK(H)$ respectively.
The map $\RR\ni t\mapsto f_t\in C_b(\RR)$ being strictly continuous, the map  $\RR\ni t\to  e^{-itP}\in \cL(H)$ is thus strictly continuous too. Specializing the semi-norms giving the strict topology to rank one operators, this means that 
$ t\longmapsto e^{-itP}x  \in C(\RR,H)$. The following properties are valid: 
\begin{equation}
 e^{-i(t+s)P}= e^{-itP}e^{-isP}
\end{equation}
and defining $(Ef)(t) = e^{-itP}f(t)$, $t\in\RR$,  for $f\in C_b(\RR,H)$ we also get: 
\begin{equation}
    E \in \cL( C_b(\RR,H) ).
\end{equation}

To further analyse  $e^{-itP}$, we introduce the sequence of Hilbert $A$-modules associated to $P$:  
$$\forall s\in \RR_{+}, \ H^{s} = \dom (1+P^{2})^{s/2} \text{ and } H^{0} = H$$
Note that $H^{k} = \dom P^{k}$ for $k\in\NN$. The Hilbertian structure of $H^{s}$ is given by:
$$ \langle u , v \rangle_s =   \langle (1+P^{2})^{s} u , v \rangle.$$
For negative order $s$, we define $H^{s}$ to be the completion of $\dom P$ with respect to the prehilbertian structure given by the scalar product above. 
We refer to this family of Hilbert $A$-modules as the \emph{intrinsic scale of Sobolev modules of $P$}. It  was introduced in \cite{Vassout2006} in the framework of groupoid $C^{*}$-algebras.

We recall that $\widetilde{\pi}(f_{t}(\lambda)\lambda^{k}) = \widetilde{\pi}(f_{t}(\lambda))\widetilde{\pi}(\lambda^{k})$ and that $\widetilde{\pi}(\lambda^{k}) = P^{k}$ for any  $k\ge 0$, therefore:
\[
e^{-itP}(H^{k}) = H^{k} \text{ and } e^{-itP}P^k =P^ke^{-itP}.
\]
In particular, we get $e^{itP}\in \cL(H^{k})$ and  $ t\longmapsto e^{-itP}x \in C(\RR,H^{k})$ for any $x\in H^k$. 
Since $\frac1t(e^{-it\lambda}-1) \overset{t\to 0}{\longrightarrow} i\lambda$ uniformly on compact subsets of $\RR$, we get  using \cite[Appendix]{deb-skand2015} that 
 $\frac1t(e^{itP}-1)$ converges to $iP$ strongly,  that is, 
 \[
   \| \frac1t(e^{itP}x-x)-iPx \|_H \overset{t\to 0}{\longrightarrow} 0 , \qquad \text{ for all } x\in H^{1}.
 \]
Therefore   
\begin{equation}
 \forall x\in H^{1},\quad  (t\longmapsto e^{-itP}x )\in C^{1}(\RR,H)\cap C^{0}(\RR,H^{1})
\end{equation}
and 
\begin{equation}
 \forall x\in H^{1},\ \forall t\in \RR, \quad  \frac{d}{d t}e^{-itP}x = -iPe^{-itP}x . 
\end{equation}
Repeating the previous arguments gives for any natural number $k$:
\begin{equation}
 \forall x\in H^{k},\quad  (t\longmapsto e^{-itP}x )\in \bigcap_{0\le j\le k} C^{j}(\RR,H^{k-j}).
\end{equation}
This eventually implies: 
\begin{equation}
 \forall x\in H^{\infty},\quad  (t\longmapsto e^{-itP}x )\in \bigcap_{k} C^{\infty}(\RR,H^{k}) =:  C^{\infty}(\RR,H^{\infty}),
\end{equation}
where $H^{\infty} = \cap_{k} H^{k}$ has Frechet space structure given by the seminorms $\| \cdot \|_{H^{k}}$, $k\ge 0$.
We can now state the result:
\begin{thm}\label{prop:Cauchy}
Let $k$ be a positive integer. For any  $f\in C^{k-1}(\RR, H^k)$  and $g\in H^k $, the Cauchy problem:
\begin{equation}\label{eq:cauchy-pb-2}
\begin{cases}
(\frac{\partial}{\partial t} +iP) u = f \\
 u(0) = g  
\end{cases}
\end{equation} 
has a unique solution in $\bigcap_{0\le j\le k} C^{j}(\RR,H^{k-j})$, given by  
\begin{equation}\label{sol-cauchy-complete}
  u(t) = e^{-itP}g + \int_0^t e^{i(s-t)P}f(s)ds.
\end{equation}
\end{thm}
\begin{proof}
That $ e^{-itP}g$ is  in the required space and satisfies the equation when $f=0$ is done before the statement of the theorem.  Straightforward arguments prove that the second term in the expression of $u(t)$ in \eqref{sol-cauchy-complete} is in the required space too, and it is then obvious that  $u$ solves \eqref{eq:cauchy-pb}. For unicity, consider the case $f=g=0$ and let $u$ be a solution.  Pairing the equation with $u$ on both sides gives the relations
\begin{align*}
   -i\langle \frac{\partial}{\partial t}u, u \rangle + \langle Pu , u \rangle & = 0 \\
    i\langle  u,\frac{\partial}{\partial t}u \rangle + \langle u , Pu \rangle & = 0. 
   \end{align*}
Since $P$ is selfadjoint, substracting both relations gives
\begin{eqnarray*}
 \frac{\partial}{\partial t} \langle u, u\rangle  =  \langle \frac{\partial}{\partial t}u, u \rangle +  \langle  u,\frac{\partial}{\partial t}u \rangle = 0.
\end{eqnarray*}
Therefore, $ \|u(t)\|_H^2 =\|\langle u(t), u(t)\rangle\|_A = \|\langle u(0), u(0)\rangle\|_A =0$ for any $t$. 
\end{proof}
Keeping the previous setting, let $B$ be a $C^{*}$-algebra,  $L$ be a Hilbert $B$-module and $\lambda : A\longrightarrow \cL(L)$ be a representation. Then $  H_{\lambda} = H \otimes_{\lambda} L$ is a Hilbert $B$-module and $ P_{\lambda} = P \otimes_{\lambda} \Id$ is a selfadjoint regular operator acting on it. Then, Proposition \ref{prop:Cauchy} applies to $P_{\lambda} $ and we get the following corollary, using  \cite[14.3.2]{skand-cours2015}.
\begin{cor}
\begin{equation}\label{eq:rep-mod}
  \text{ for } k > 0, \quad \dom P_{\lambda}^{k} = H^{k}_{\lambda}
\end{equation}
and we have the equality:
\begin{equation}\label{eq:rep-exp}
    e^{itP_{\lambda}} =  e^{itP}  \otimes_{\lambda} \Id. 
\end{equation}
\end{cor}

\section{Distributions, test functions and weak factorizations for a Lie groupoid}\label{sec:dist-test-facto}

 From now on, and in the remaining parts of this article,  we fix a  Lie groupoid $G$ of dimension $n = n^{(1)} + n^{(0)}$ with compact basis $G^{(0)}=M$ of dimension $n^{(0)}$.
We recall that:
 $$ \Omega^{1/2} := \Omega^{1/2}_{s}\otimes \Omega^{1/2}_{r} = \Omega^{1/2}(r^{*}AG) \otimes\Omega^{1/2}( s^{*}AG) \simeq \Omega^{1/2}(\ker ds) \otimes \Omega^{1/2}(\ker dr).$$ 
and that the bundle $\Omega^{1/2}_{0}$  used in the space of test functions $ \cD(G) = C^{\infty}_{c}(G, \Omega^{1/2}_{0})$ satisfies:
\begin{equation}
\Omega^{1/2}_{0} := \Omega^{-1/2}\otimes\Omega^{1}_{G} \simeq \Omega^{1/2}_{s}\otimes\Omega^{-1/2}_{r}\otimes s^{*}\Omega_{M} \simeq  \Omega^{1/2}_{r}\otimes\Omega^{-1/2}_{s}\otimes r^{*}\Omega_{M} \simeq   r^{*}\Omega^{1/2}_{M}\otimes s^{*}\Omega^{1/2}_{M} .
\end{equation}  
All the isomorphisms above are easily checked using the isomorphisms
$$\Omega^{\alpha}_{G} \simeq \Omega^{\alpha}_{s}\otimes s^{*}\Omega^{\alpha}_{M} \simeq  \Omega^{\alpha}_{r}\otimes r^{*}\Omega^{\alpha}_{M}$$ 
that result from the exact sequences: 
$$ 0 \To \ker d\sigma \To TG  \overset{d\sigma}\To \sigma^{*}TM \To 0, \qquad \sigma =s,r $$
as well as straight properties of the calculus of densities. To finish with this description, we mention that $\Omega^{1/2}_{0}$ is related, but distinct, to the transverse density bundle $\cD^{\tr}_{AG}$ of \cite{Crainic-Mestre2019}. The latter is $G$-invariant and serves to produce geometric transverse measures useful for the geometry of groupoids and stacks, while our choice of \enquote{transverse} bundle is  required for the pairing with densities in $\Omega^{1/2}$, but only equivariant with respect to $\RR$-actions provided by invariant vectors fields.

Moreover, besides its pairing with distributions, the space $\cD(G)$ appears to be a bimodule over $ C^{\infty}_{c}(G)$, with left and right multiplication given by the canonically defined integrals:
\begin{equation}
f \star \xi (\gamma) = \int_{G_{s(\gamma)}}f(\gamma\alpha^{-1})g(\alpha) \text{ and }  \xi \star f(\gamma) = \int_{G_{s(\gamma)}}\xi (\gamma\alpha^{-1}) f(\alpha).
\end{equation}
 Finally, we recall that the embedding $C^{\infty}_{c}(G)\subset \cD'(G)$ is given  by:
\begin{equation*}
 \forall u \in C^\infty_c(G),\omega\in \cD(G), \quad \langle u, \omega \rangle = \int_G u(\gamma)\omega(\gamma)d\gamma.
 \end{equation*}
The inversion map $\iota : G\to G$ acts on sections of $\Omega^{1/2}$ and $\Omega^{1/2}_{0}$ in the natural way. This gives involutive isomorphisms:
\begin{equation}\
\iota^{*} : \cD(G) \To \cD(G) \text{ and } \iota^{*} : C^{\infty}_{c}(G) \To C^{\infty}_{c}(G).
\end{equation}
The second one extends to an involutive isomorphism $\iota^{*}:\cD'(G) \To \cD'(G)$.
\begin{prop}\
\begin{enumerate}
\item
For any $(u,\omega)\in \cD'_{r,s}(G)\times \cD(G)$, we have 
\begin{equation}
\langle u, \omega \rangle =   \langle \delta_{M},  \iota^{*}u \star \omega \rangle = \langle \delta_{M}, \omega \star \iota^{*}u\rangle  = \langle \iota^{*}u, \iota^{*}\omega \rangle  \text{ (trace property). }
\end{equation}

The trace property $\langle u, \omega \rangle = \langle \iota^{*}u, \iota^{*}\omega \rangle $ is still valid with $u\in \cD'(G)$.  
\item The map $\iota^{*}$ is an anti-isomorphism of the algebra $\cE'_{r,s}(G)$: 
\begin{equation}\label{eq:iota-anti-iso}
\forall u,v \in \cE'_{r,s}(G), \quad  \iota^{*} (u \star v) = \iota^{*}v \star \iota^{*}u,
\end{equation}
\item  The space $\cD(G)$ is a bimodule over $\cE'_{r,s}(G)$ and $\iota^{*}$ is a bimodule antisomorphism: 
\begin{equation*}
\forall u,v \in \cE'_{r,s}(G), \forall \omega\in \cD(G), \quad  \iota^{*} (u \star \omega \star v) = \iota^{*}v \star \iota^{*}\omega \star  \iota^{*}u,
\end{equation*}
\item For any $u,v \in \cE'_{r,s}(G)$ and $\omega\in \cD(G)$, we have:
\begin{align*}
 \langle u \star v, \omega \rangle & =  \langle  v, \iota^{*}u \star \omega \rangle =  \langle  u ,   \omega \star \iota^{*}v \rangle =   \langle  \delta_{M},   \iota^{*}u \star \omega \star \iota^{*}v  \rangle
\end{align*}
\end{enumerate}
\end{prop}
\begin{proof}
That $\cD(G)$ is a bimodule over   $\cE'_{r,s}(G)$ follows directly from \cite{Androulidakis-Skandalis,LMV1}. If $u$ is $C^{\infty}$, the quantity $\langle u, \omega \rangle$ is the integral of the one density on $G$ defined by  the product $u\omega$, whose integral is then invariant by action of diffeomorphisms. In particular, $\langle u, \omega \rangle = \int_{G} u(\gamma^{-1})\omega(\gamma^{-1})$. On the other hand, one is allowed to write 
$$ \langle u, \omega \rangle = \int_{M} \big(\int_{G_{x}} \iota^{*}u(\gamma^{-1})\omega(\gamma)\big) dx = \int_{M}  \iota^{*}u \star \omega (x) dx .$$
Both identities together give (1) when $u$ is $C^{\infty}$, and the general case follows by density and continuity. 
  The identities given in (2) and (3) are then checked easily. 
\end{proof}

Let  $X\in \Gamma(AG)$. Since $AG \subset TG$, the vector field $X$ provides at any $x\in M$ a local derivation $X_{x}: \cD(G)\to \Omega^{1}(T_{x}M)$ and $x\mapsto X_{x}\omega$ is $C^{\infty}$ for any $\omega\in\cD(G)$. Therefore $X\in \Gamma(AG)$ provides a distribution 
$$\tau_{X}\in \Diff(G)=\{ u\in \Psi^{*}_{G}\ ; \ \supp u \subset M \} \subset \Psi^{*}_{c}(G),$$

via the formula:
\begin{equation*}
\forall \omega\in C^{\infty}(G), \quad \langle \tau_{X} , \omega\rangle = \int_{M} X\omega. 
\end{equation*}
We recall that the algebra isomorphism
\begin{alignat}{2}\label{eq:iso-dist-Gop}
\op : & & (\cE'_{r,s}(G), \star) & \To (\Op_{G}^{*,c}, \circ) \\
        & &     u  & \longmapsto u\star \cdot \notag
\end{alignat}
maps $ \Psi^{*}_{G,c}$ (resp. $\Diff^{*}_{G}$) to the algebra of uniformly supported and equivariant $C^{\infty}$ family of pseudodifferential (resp.  equivariant $C^{\infty}$ family of differential) operators on the fibers of $s$ \cite{NWX,Monthubert2003,LMV1}.

Note that the action of $\tau_{X}$ as a differential $G$-operator is given, up to inversion, by the right invariant vector field $\widetilde{X}$ associated with $X$:
\begin{equation*}
\forall u \in C^{\infty}_{c}(G), \quad  \iota^{*}\tau_{X} \star u = \widetilde{X}u.
\end{equation*}
Let  $\varphi$ be the flow of the vector field  $\widetilde{X}$. By compacity of $M$, there exists $ \varepsilon > 0$ and a neighborhood $U$ of $M$ into $G$  such that $\varphi$ is defined on $] - \varepsilon, \varepsilon [\times U$. Since $\widetilde{X}_{\gamma}=dR_{\gamma}(X_{r(\gamma)})$ for any $\gamma$ we get the relation $\varphi(t,\gamma\eta)=\varphi(t,\gamma)\eta$ whenever  both terms are well defined. Therefore the flow $\varphi$ is well defined on   $] - \varepsilon, \varepsilon [\times G$, and then on $\RR\times G$ using the one parameter group property. This proves that  the flow of $\widetilde{X}$ is complete and commutes with right multiplication in $G$:
\begin{equation}
\forall t\in\RR, \forall \gamma,\eta\in G^{(2)},\quad  \varphi(t,\gamma)\in G_{s(\gamma)}\text{ and } \varphi(t,\gamma\eta) = \varphi(t,\gamma)\eta.
\end{equation}
In other words, $X$ provides an action of $\RR$ on the \emph{manifold} $G$, which is equivariant with respect to right multiplication. 
Also, the map $\psi := r\circ \varphi : \RR\times M \To M$ is the flow of the vector field $\mathfrak{a}(X)\in \Gamma(TM)$ where $\mathfrak{a} = d r\vert_{TM}$ is the anchor map of $G$ \cite{Mackenzie2005} and the map: 
$$ \varphi : \RR \ltimes_{\psi} M \To G, \quad (t,x) \longmapsto \varphi(t,x) $$
is a ($C^{\infty}$) homomorphism of groupoids over $M$. We recall that a groupoid homorphism $h : G_{1}\to G_{2}$ over (the identity map of) $X = G_{1}^{(0)} =  G_{2}^{(0)}$ is a map satisfying $h(\alpha\beta)=h( \alpha) h(\beta)$ whenever it makes sense and $r\circ h = r$, $s\circ h = s$. 

We record the following simple fact:
\begin{prop}\label{lem:push-star}
Let $G,H$ be two Lie groupoids with same units space $M$.
\begin{enumerate}
\item Let $h : G\To H$ a $C^{\infty}$ be a homomorphism over $M$. Then the pushforward map $h_{!}$ gives rise to a (unital, involutive) algebra homomorphism:
\begin{equation}
h_{!} : \cE'_{r,s}(G) \To \cE'_{r,s}(H).
\end{equation}
\item Let $h_{1},h_{2} : G\To H$ be two $C^{\infty}$ homomorphisms over $M$ and set $h_{12} := m\circ(h_{1}\otimes h_{2}) : G^{(2)}\To H$. Then for any $u,v\in \cE'_{r,s}(G)$, we have 
\begin{equation}
h_{1!}u \star h_{2!}v = h_{12!}( u\otimes v\vert_{G^{(2)}}). 
\end{equation}
\end{enumerate}
\end{prop}
\begin{proof}
 First of all, $h_{!} :  \cE'(G) \To \cE'(H)$ is well defined. Indeed, if $\omega\in \cD(H)$  and $u \in C^{\infty}_{c}(G)$, then $\omega(h(\gamma))\in \Omega^{1/2}_{M,r(h(\gamma))}\otimes \Omega^{1/2}_{M,s(h(\gamma))} = \Omega^{1/2}_{M,r(\gamma)}\otimes \Omega^{1/2}_{M,s(\gamma)}$ and therefore:
$$ \langle h_{!}u, \omega \rangle := \langle u, \omega\circ h \rangle = \int_{G} u(\gamma) \omega(h(\gamma)) $$
is canonically defined. The algebraic remaining assertions come from the identities: $m\circ(h\otimes h) = h\circ m$ on $G^{(2)}$, $h\iota= \iota h$ on $G$, from the functoriality of pushforwards: $f_{!}g_{!}=(fg)_{!}$, and from the definition of the convolution product of distributions: $u\star v = m_{!}(u\otimes v\vert_{G^{(2)}})$. 
\end{proof}
  
The goal now is to export to Lie groupoids (with compact unit spaces) a classic result by Dixmier and Malliavin about Lie groups \cite{Dix-Mal78}. This will be the main technical tool used to embed reduced $C^{*}$-algebras into distributions. 
\begin{thm}
Let $V$  be an open neighborhood of $M$ into $G$ and $\omega\in \cD(G)$. Then $\omega$ is a finite sum of elements: 
\begin{equation}\label{eq:weak-fact}
 \xi \star \chi 
\end{equation}
where $\xi \in C^{\infty}_{c}(G)$,   $\supp{\xi}\subset V$ and $\chi\in \cD(G)$, $\supp{\chi}\subset \supp \omega$.
The result is still valid with the factors flipped in the convolution above. 
\end{thm}
We adapt the proof of \cite[Theorem 3.1]{Dix-Mal78} to groupoids. Firstly, \cite[Lemma 2.5 and Remark 2.6]{Dix-Mal78} gives rise to: 
\begin{lem}\label{lem:step1}\
Let  $X\in \Gamma(AG)$ and $\varphi$, $\psi$  be the associated actions of $\RR$ on $G$ and $M$.  Let  $ \varepsilon > 0$. For any test function $\omega\in \cD(G)$, there exists $a_{1},b_{1}\in C^{\infty}_{c}(] - \varepsilon, \varepsilon [)\subset \cE'_{r,s}(\RR\ltimes_{\psi}M)$ and $\omega_{1} \in\cD(G)$ with $\supp{\omega_{1}}\subset \supp \omega$ such that
\begin{equation}\label{eq:step1}
\omega = \varphi_{!}a_{1} \star \omega_{1} + \varphi_{!} b_{1}\star \omega. 
\end{equation}
\end{lem}
\begin{proof}[Proof of the Lemma]
First of all, we pick up a sequence $(p_{j})$ of semi-norms characterizing the topology of $\cD(\supp \omega)$, and set $\beta_{k}=k^{-2}\inf\{ (p_{j}(D_{X}^{2i}\omega)+1)^{-1}\ ; \  i,j \le k \}$. Then the series $\sum (-1)^{k}\alpha_{k}D_{X}^{2k}\omega$ converges in $\cD(G)$ for any sequence $ 0\le \alpha_{k} \le \beta_{k}$. 
 Next, we choose by \cite[Lemma 2.5 and Remark 2.6]{Dix-Mal78}, two functions $a_{1},b_{1}\in C^{\infty}_{c}(] - \varepsilon, \varepsilon [)$  and a sequence $0 \le \alpha_{k}\le \beta_{k}$ such that 
\begin{equation}\label{eq:split-delta}
 \delta =  a_{1} \star \sum_{k=0}^{\infty} (-1)^k\alpha_k \delta^{(2k)}  + b_{1}= \sum_{k=0}^{\infty}(-1)^{k}\alpha_{k}a_{1}^{(2k)} +b_{1}  \text{ in } \cE'(\RR) \subset \cE'_{r,s}(\RR\ltimes_{\psi} M)
\end{equation}
Now (1) of Proposition \ref{lem:push-star} gives the identity \eqref{eq:step1}, with $\omega_{1} =\sum_{k=0}^{\infty} (-1)^{k}\alpha_{k}D_{X}^{2k}\omega$.
\end{proof}

\begin{proof}[Proof of the theorem]

Let $X_{1},\ldots,X_{\ell}$ be a family generating the $C^{\infty}(M)$-module $\Gamma(AG)$, and 
$$\varphi_{i} : G_{i} := \RR\ltimes_{\psi_{i}} M \to G$$
be  the associated homomorphisms.
Applying the lemma to $\omega$ with $\varphi=\varphi_{1}$, we get 
\begin{equation}\label{eq:magic-G-step1-again}
\omega =  \lambda_1  \star  \omega_{1}   +  \mu_1 \star \omega \text{ in } \cE'_{r,s}(G),
\end{equation}
with $\lambda_1,\mu_{1} $ in  $ \varphi_{1!}(C^{\infty}_{c}(] - \varepsilon, \varepsilon [))$. Applying the lemma  to $\omega$ and $\omega_{1}$ with $\varphi_{2}$  we get, with intuitive notation:
\begin{equation}\label{eq:magic-G-step2-1-again}
 \omega_{1} =   \lambda_{2,1}  \star \omega_{2,1}   +  \mu_{2,1} \star  \omega_{1}  \quad ; \quad \omega =   \lambda_2  \star \omega_{2}   +  \mu_2 \star \omega  \text{ in }  \cE'_{r,s}(G) .
\end{equation}
Inserting \eqref{eq:magic-G-step2-1-again} into \eqref{eq:magic-G-step1-again}, we get: 
\begin{equation}\label{eq:magic-G-step2}
\omega =   \lambda_1  \star   \lambda_{2,1}  \star  \omega_{2,1} +   \lambda_1  \star  \mu_{2,1} \star\omega_{1}  +  \mu_1 \star \mu_2  \star  \omega_{2}   +  \mu_1 \star  \mu_2 \star \omega \text{ in }  \cE'_{r,s}(G) ,
\end{equation}
where all the $\lambda_{j\bullet},\mu_{j\bullet}$ are in the range of  $ C^{\infty}_{c}(] - \varepsilon, \varepsilon [)$ by $\varphi_{j!}$, $j=1,2$, and all $\omega_{\bullet}$ are test functions with support in $\supp \omega$. 

Repeating the argument with  $\varphi_{3!}$, \ldots,  $\varphi_{\ell!}$ we get that $ \omega$ is equal to a sum of $2^{\ell}$ distributions of the form: 
\begin{equation}\label{eq:magic-G-stepl}
  \xi_{1} \star  \xi_{2} \star \cdots  \star \xi_{\ell } \star \chi 
\end{equation}
where $\xi_{j} = \varphi_{j!}(k_{j})\in \cD'(G)$ for some $k_{j}\in C^{\infty}_{c}(]- \varepsilon, \varepsilon [)$ and $\chi \in \cD(G)$ with  $\supp \chi \subset \supp \omega$. 
Setting as in Proposition \eqref{lem:push-star}:
\begin{equation}
\varphi = \varphi_{1\cdots\ell}:  \RR^{\ell}\times M\simeq G_{1}\underset{M}\times \cdots \underset{M}\times G_{\ell}  \longrightarrow G 
\end{equation}
and after an obvious induction, we get 
 \begin{equation}\label{eq:is-it-Cinfty}
  \xi_{1}\star \xi_{2}   \star \cdots  \star  \xi_{\ell}  = \varphi_{!} k\quad  \text{ with }\quad  k = k_{1}\otimes \cdots \otimes k_{\ell} \in C^{\infty}_{c}(]- \varepsilon, \varepsilon [^{\ell}).
\end{equation}
Since $\partial_{t_{j}}\varphi(t,x)\vert_{t=0} = X_{j}(x)$ and $\varphi(0,x) = x$  for any $1\le j \le \ell$ and $x\in M$, we get that $\varphi$ is a submersion on  $]-\varepsilon,\varepsilon[^{\ell}\times M$ if $ \varepsilon >0$ is small enough. Since the push forward of a $C^{\infty}$ distribution by a submersion is $C^{\infty}$, we get that $\varphi_{!} k$:
\begin{equation}\label{eq:it-is-Cinfty}
 \forall \eta \in \cD(G),\quad \langle  \varphi_{!} k , \omega \rangle = \int_{\RR^{\ell}\times M} k(t)\eta(\varphi(t,x)) dt dx,
 \end{equation}
 is $C^{\infty}$ and supported in  $\varphi(]-\varepsilon,\varepsilon[^{\ell}\times M)$. Taking $ \varepsilon> 0$ small enough ensures that this last set is contained in $V$. 
  \end{proof}

\section{Embedding $C^{*}_{r}(G)$ into $\cD'(G)$ and regularizing operators} \label{sec:Sobolev}
 From now on and in the remaining parts of this article, we fix a compactly supported, first order elliptic pseudodifferential $G$-operator  $P\in \Psi^{1}_{G,c}$ and  we denote by $C^{*}_{r}(G)$  the reduced $C^{*}$-algebra of $G$.  

\begin{thm}
There is a continuous embedding:
\begin{equation}
C^{*}_{r}(G) \hookrightarrow \cD'(G)
\end{equation}
that extends the  pairing:
\begin{equation}
\forall u \in C^{\infty}_{c}(G), \ \forall \omega\in \cD(G), 
\quad \langle u\ , \ \omega \rangle = \int_{M} \iota^{*}u\star\omega
\end{equation}
\end{thm}
\begin{proof}
Let $ v,\xi \in C^{\infty}_{c}(G)$ and $\chi\in \cD(G)$. We have: 
\begin{align}\label{eq:Cstar-dist-1}
 \langle v , \xi  \star  \chi \rangle &  =   \langle \iota^{*}\xi \star  v , \chi  \rangle  \notag\\
  &  =    \langle \iota^{*} v \star \xi ,   \iota^{*}\chi  \rangle  \text{ (trace property)}  \notag\\
  & =  \int_{G}  \iota^{*} v\star\xi(\alpha) \iota^{*}\chi(\alpha).
 \end{align}
 Let $\mu$ and $\mu_{0}$ be positive sections of, respectively, the degree $1$ densities bundles of $AG$ and $TM$. We define 
 $\mu_{r}\in C^{\infty}(G, \Omega^{1}_{r})$ and  $\mu_{s,0}\in C^{\infty}(G,s^{*}\Omega^{1}_{M})$ by 
 \begin{equation}
 \mu_{r}(\gamma)= \mu(s(\gamma)) \quad \text{ and } \quad \mu_{s,0}(\gamma)=\mu_{0}(s(\gamma)).
 \end{equation}
 We observe: 
 $$ (\iota^{*} v\star\xi) .\mu^{-1/2}_{r} =  \iota^{*} v\star \xi' \in C^{\infty}_{c}(G,\Omega^{1/2}_{s}) \text{ with } \xi' = \xi \mu^{-1/2}_{r} \in C^{\infty}_{c}(G,\Omega^{1/2}_{s})$$
 and 
 $$ \chi' = \chi .\mu^{1/2}_{r}.\mu^{-1}_{s,0} \in C^{\infty}_{c}(G,\Omega^{1/2}_{s}). $$
 This allows us to write
 \begin{align}\label{eq:Cstar-dist-intermed}
 \langle v , \xi  \star  \chi \rangle = \int_{G}  \iota^{*} v\star\xi(\alpha) \iota^{*}\chi(\alpha) & =  \int_{M}\Big( \int_{G_{x}} \iota^{*} v\star\xi'(\alpha) \iota^{*}\chi'(\alpha)\Big)d\mu_{0}  
 \end{align}
 and to use the Cauchy Schwarz inequalities for the Hilbert spaces $(L^{2}(G_{x},\Omega^{1/2}_{G_{x}}), \|\cdot\|_{x})$ in the following computations:
\begin{align}\label{eq:Cstar-dist-2}
  \vert \langle v , \xi  \star  \chi \rangle \vert & \le  \int_{M}d\mu_{0}. \sup_{x \in M} \vert \int_{G_{x}} \iota^{*} v\star\xi'(\alpha) \iota^{*}\chi'(\alpha) \vert   \notag \\
  & \le  c_{M}\sup_{x \in M} \| \iota^{*} v\star\xi' \|_{x}  \| \iota^{*}\chi' \|_{x}   \notag \\
  & \le  c_{M}  \| \iota^{*} v  \|_{C^{*}_{r}(G)}  \| \xi' \|_{L^{2}_{s}(G)}  \| \iota^{*}\chi' \|_{L^{2}_{s}(G)}  \notag\\
  & \le  c  \| v  \|_{C^{*}_{r}(G)}
\end{align}
Now let  $\omega \in \cD(G)$  and pick up a weak factorisation $\omega = \sum_{j} \xi_{j}\star \chi_{j}$. Let $ u \in C^{*}_{r}(G)$ and  choose a sequence $(u_{k})$ with $u_{k}\in  C^{\infty}_{c}(G)$ and $u_{k}\to u$ in $C^{*}_{r}(G)$. Using the previous estimates, we see that the sequence $\langle u_{k}, \omega\rangle\in \CC$ satisfies the Cauchy criterium and thus converges. Setting $u(\omega)=\lim_{k\to +\infty}\langle u_{k}, \omega\rangle$ with get that $u\in\cD'(G)$ and that $u_{k}\to u$ in $\cD'(G)$. 
\end{proof}
We now give some complements to the properties of the regularizing operators:
\begin{align}
  \Psi^{-\infty}_{G} & :=\{ R \in \cL(C^*_r(G))\ ;\  R\in \cL(H^s,H^t) \text{ for all } s,t\in\NN\} \\
                             & = \{ R \in \cL(C^*_r(G))\ ;\  P_1RP_2\in \cL(C^*_r(G) ) \text{ for all }P_j\in \Psi^{s_j}_{G,c},\   s_j\in\NN,\ j=1,2\} .
 \end{align}
introduced exaclty in this form in \cite{Vassout2006} and in an equivalent form in \cite{LMN2005}. In both references, this ideal of the $C^{*}$-closure of $\Psi_{G,c}$ is proved to be stable under holomorphic functional calculus. 
Here  $H^{s}$ denotes the  scale of intrinsic Sobolev $C^{*}_{r}(G)$-modules. 
\begin{prop} 
Operators in $\Psi^{-\infty}_{G}$ are exaclty convolution operators by elements of  $\cH^{\infty}$. In other words, as subsets of the multipliers algebra  $\cM(C^*_r(G))$, these sets coincide:
$$\Psi^{-\infty}_{G} = \cH^{\infty} \subset \cM(C^*_r(G)). $$
\end{prop}
\begin{proof}
We know that $\Psi^{-\infty}_{G} \subset \cK(C^*_r(G))=C^*_r(G)$. Let $T \in \Psi^{-\infty}_{G}$. For any $k\in \NN$, we have:
$$ (1+P^2)^{k} T  = S_k \in C^*_r(G) \text{ and } T(1+P^2)^{k}  = S'_k \in C^*_r(G) $$
Then $T = (1+P^2)^{-k} S_k =S'_k (1+P^2)^{-k} \in  H^{2k}\cap (H^{2k})^{*}$ for any $k$, which proves the first inclusion. The second one is obvious.
\end{proof}

All the previous statements hold true for the maximal $C^{*}$-algebra of $G$ but we stay in the framework of the reduced $C^*$-algebra, because the embedding $C^{*}_{r}(G)\hookrightarrow \cD'(G)$ and the regular representation allow us to precise in what extent elements of $\Psi^{-\infty}_{G} = \cH^{\infty}$ are \emph{regularizing}.  For that purpose, we let $\Psi^{*}_{G}$ act not on the scale of intrinsic Sobolev modules $H^{s}$, but on their representation via the left regular representation. These $C(M)$-modules are concretely given as follows, for $k\in \ZZ$:
\begin{equation}
H^{k}_{s} =  \overline{C^{\infty}_{c}(G,\Omega^{1/2}_{s})}^{\langle \cdot  | \cdot \rangle_{k,s}}, \quad \langle \omega \ | \ \eta \rangle_{k,s}  =   \langle (1 + P^{2})^{k} \star \omega \ | \ \eta\rangle_{s}  \in C(M).
\end{equation}

\begin{lem} We have: 
\begin{equation}
H^{\infty}_{s} := \bigcap_{k\in\ZZ} H^{k}_{s} \subset C^{\infty,0}_{s}(G,\Omega^{1/2}_{s}) \text{ and }  H^{-\infty}_{s} := \bigcup_{k\in\ZZ} H^{k}_{s} \supset \cE'_{s,0}(G,\Omega^{1/2}_{s}).
\end{equation}
\end{lem}

\begin{proof}[Proof of the lemma]
 Let $ \omega \in H^\infty_{s} $. Since pointwise multiplication operators by compactly supported $C^{\infty}$ functions are in $\Mor(H^k_{s})$ for any $k$, we can 
assume that $\omega$ is compactly supported in the domain $U$ of a local trivialization  $\kappa : U\to \RR^{n^{(1)}}\times \RR^{n^{(0)}}$, $\kappa(x)=(x',x'')$  of the submersion $s$.  By assumption, we have 
\begin{equation}
\forall k\ge 0, \quad \Delta^k_G \omega\in C^0_c(\RR^{n^{(0)}}, L^2(\RR^{n^{(1)}})). 
\end{equation}
Here $\Delta_G=d^*d\in \Diff^2(G)$ is the Laplacian associated with a given euclidean structure on $AG$. The ellipticity of  each term of the $C^\infty$ family $(\Delta_{G,x''})_{x''\in\RR^{n^{(0)}}}$ and the compactness of  $\supp \omega$ imply using usual Garding inequality that  $\omega \in C^0(\RR^{n^{(0)}}, H^{2k}( \RR^{n^{(1)}}))$ for any $k$,  where $H^*$ denotes here the usual Sobolev spaces of euclidean spaces. We then conclude that 
$\omega\in C^0(\RR^{n_0}, C^\infty(\RR^{n^{(1)}})) = C^{\infty,0}_{\pr_2}(\RR^{n^{(1)}}\times \RR^{n^{(0)}})$. This proves $H^{\infty}_{s} \subset C^{\infty,0}_{s}(G,\Omega^{1/2}_{s})$. \\
Let $u \in \cE'_{s,0}(G,\Omega^{1/2}_{s})$. The result \cite[Theorem 4.4.7]{Horm-classics} extends immediately to continuous family of distributions so  there exists $k\in \NN$  and  finite collections: $u_{I}\in C_{c}(G,\Omega^{1/2}_{s})$, $D_{I}\in \Diff^{k}_{G}$ such that 
\begin{equation}
 u = \sum_{I} D_{I}u_{I}. 
\end{equation}
Since $C_{c}(G,\Omega^{1/2}_{s}) \subset L^{2}_{s}(G)$ and $(1+P^{2})^{-k/2}D_{I}\in \cL(L^{2}_{s}(G))$, we then conclude that $u \in H^{-k}_{s}$.
\end{proof}

\medskip We recall  \cite{Moerdijk2003,Mackenzie2005} that for any $x\in M$, the orbit $\cO=r(s^{-1}(\{x\}))\subset M$ is an immersed submanifold, the map $r : G_{x} \To \cO$ is a submersion (actually a $G_{x}^{x}$ principal bundle) and that $G_{\cO} \rightrightarrows \cO$ is an immersed subgroupoid. We set: 
\begin{equation}
C^{\infty,0}_{\orb}(G,E) = \{ u \in C(G,E) \ ; \ \forall x\in M, \forall D\in\Diff(G),   Du \in C^{\infty}(G_{\cO},E),\ \cO=r(s^{-1}(\{ x\}) \}. 
\end{equation}

\begin{thm}
The following inclusions hold true:
\begin{equation}
C^{\infty,0}_{\orb}(G) \cap C_{c}(G) \subset \cH^{\infty} \subset  C^{\infty,0}_{\orb}(G)\cap C^{*}_{r}(G).
\end{equation}
\end{thm}
In particular, since $\cH^{\infty}$ is an ideal in $C^{*}_{r}(G)$:
\begin{cor}
Any $h \in \cH^{\infty}$ provides continuous operators : 
\begin{equation*}
 h  \star \cdot :  C^{*}_{r}(G) \To C^{\infty,0}_{\orb}(G)\cap C^{*}_{r}(G) \text{ and }  \cdot \star\, h :  C^{*}_{r}(G) \To C^{\infty,0}_{\orb}(G)\cap C^{*}_{r}(G).
\end{equation*}
\end{cor}
\begin{proof}[Proof of the theorem]
 Let $u \in \cH^{\infty}$. By  \cite{Vassout2006} and the left regular representation,  $u$ maps  $H^{-k}_{s}\to H^{k}_{s}$ continuously  for any $k\in \NN$. Therefore, the previous lemma implies that $u$ maps  $ \cE'_{s,0}(G,\Omega^{1/2}_{s}) \to C^{\infty,0}_{s}(G,\Omega^{1/2}_{s})$ continuously. In particular for every $x\in M$, the distribution $\kappa_{x}(\gamma_{1},\gamma_{2})=u(\gamma_{1}\gamma_{2}^{-1})\in \cD'(G_{x}\times G_{x},\Omega^{1/2}_{x})$ extends to a continuous map:
 \begin{equation}
 \kappa_{x} : \cE'(G_{x},\Omega^{1/2}_{G_{x}}) \to C^{\infty}(G_{x},\Omega^{1/2}_{G_{x}}) 
 \end{equation}
 which implies that $\kappa_{x}\in C^{\infty}(G_{x}\times G_{x},\Omega^{1/2}_{x})$ for fixed $x$. Next, consider $x\in M$, $\cO=r(s^{-1}(\{x\}))$ the orbit of $x$ in $M$ and fix $(\gamma_{1},\gamma_{2})\in G_{x}\times G_{x}\subset G\underset{s}\times G$. We denote by $\pi :  G\underset{s}\times G \To M$ the obvious submersion. Since $r : G_{x} \To \cO$ is a submersion, there exists a $C^{\infty}$ local section $\eta : U\in y \mapsto \eta_{y}\in G_{x}$ of $r$ such that $\eta_{x}=x$, defined on some open neighborhood  $U$ of $x$ into $O$. Then $V = \pi_{\cO}^{-1}(U) = \{ (\eta_{1},\eta_{2}) \ ; \ s(\eta_{1})=s(\eta_{2})\in U\}$ is an open neighborhood of $(\gamma_{1},\gamma_{2})$ into 
 $G_{\cO}\underset{s}{\times} G_{\cO}$ and we have:
 \begin{equation}
 \forall (\eta_{1},\eta_{2})\in V, \quad \kappa_{y}(\eta_{1},\eta_{2}) = \kappa_{x}(R_{\eta_{y}}\eta_{1},R_{\eta_{y}}\eta_{2}) 
 \end{equation}
which proves that $\kappa$ is $C^{\infty}$ on $G_{\cO}\underset{s}\times G_{\cO}$, and thus that $u$ is $C^{\infty}$ on $G_{\cO}$. 
It is clear that $C^{\infty,0}_{\orb}(G) \cap C_{c}(G)$ is  contained in $C^{*}_{r}(G)$ and is invariant under the left and right convolution by $P$. The inclusion $ C^{\infty,0}_{\orb}(G) \cap C_{c}(G) \subset \cH^{\infty}$ follows.
\end{proof}
Summarizing the content above, we have proved that regularizing operators are actually convolution operators by distributions on $G$ lying in the class $\cH^\infty$, the latter class being included in the class  of functions that are continuous on $G$ and infinitely differentiable over any orbit,  and thus in particular along the fibers of $s$ and $r$. Closely related results were obtained in \cite{LMN2005} under the assumption of bounded geometry for $G$. 
In the following sections, we are going to prove that $E(t)=e^{itP}$ is  a family $\RR\times G$-FIO \cite{LV1}, modulo such regularizing operators. 
 
\section{Principal and Subprincipal symbols of $G$-PDOs}\label{sec:subprinc}

As a conormal distribution, any  element of $\Psi^m_{G}= I^{m+(n^{(1)}-n^{(0)})/4}(G,M;\Omega^{1/2})$ has a principal symbol \cite[Theorem 18.2.11]{Horm-classics} in:
\begin{equation}
 S^{[m+(n^{(1)}-n^{(0)})/4+n/4]}(A^*G, \Omega^{1/2}_{A^{*}G} \otimes \hat{\Omega}^{1/2}  \otimes \hat{\Omega}^{-1/2}_{G} ).
\end{equation}
The density bundle above is canonically trivial:
\begin{equation}\
 \Omega^{1/2}_{A^{*}G} \otimes \hat{\Omega}^{1/2}  \otimes \hat{\Omega}^{-1/2}_{G} =  \Omega^{1/2}(TM\oplus A^{*}G)\oplus \Omega^{1}(AG) \otimes \Omega^{-1/2}(TM\oplus AG) \simeq M \times \CC,
\end{equation}
and since half densities on $A^*G$ contribute with a value of $n^{(1)}/2$ to the degree of symbols, the simplification above lowers the degree by the same value. In conclusion the principal symbol map is a well defined map: 
\begin{equation}
   \sigma_0: \Psi^m_{G}  \longrightarrow S^{[m]}(A^*G).
\end{equation}
Alternatively, given $\nP\in  \Psi^m_{G} $, one may consider  the family $\widetilde{\nP} =  (\nP_{x})_{x\in M}$,  $\nP_{x}\in \Psi^{m}(G_{x}, \Omega^{1/2}_{G_{x}})$ associated with $\nP$ by the isomorphism \eqref{eq:iso-dist-Gop} and collect the family of principal symbols $\sigma(\nP_{x}) \in S^{[m]}(T^*G_{x})$ into the element $\sigma(\nP)\in S^{[m]}(T^*_{s}G)$, where $T^*_sG=(\ker ds)^*$, defined by: 
$$\sigma(\nP)(\gamma,\xi)= \sigma(\nP_{s(\gamma)})(\gamma,\xi).$$
 In this point of view, the principal symbol is a map : 
\begin{equation}
   \sigma: \Psi^m_{G}  \longrightarrow S^{[m]}(T_s^*G).
\end{equation}
Both notions are related by: 
\begin{prop}\label{prop:princ-symbol-2pov}
 With the notation above, the following identity holds true:
\begin{equation}
  \sigma  = \sigma_0 \circ \rg.
\end{equation}
\end{prop}
\begin{rmk}
 Strictly speaking, the target map $\rg$ is defined on $T^*G$. It is by construction the composition of the natural restriction map $T^*G\to T^*_sG$ with  the natural map $ T^*_sG\to A^*G$. It is understood in the Proposition above that $ \rg $ means the latter. 
\end{rmk}

\begin{proof}
Let $\nP\in \Psi^m_{G}$.  Without loss of generality, we can assume that $\nP$ is supported in a local chart $U$ around some point of $M$ and satisfying:
\begin{enumerate}[-]
 \item  the local coordinates trivialize the source map, that is   $\gamma = (x'',x')$ with $s(\gamma)=x''$ on $U$,
 \item  the domain $U$ is invariant for the inversion map :  $U^{-1} = U$. 
\end{enumerate}
We then pick up a positive one density $\mu$ on $AG$ such that:
$$ \forall x\in U\cap G^{(0)}, \quad \mu(x) = | dx'|,$$
and define $\mu_{s}\in C^{\infty}(G,\Omega^{1/2}_{s})$, $\mu_{r}\in C^{\infty}(G,\Omega^{1/2}_{r})$ by:
$$ \mu_{s}(\gamma)= \mu(r(\gamma)) \text{ and } \mu_{r}(\gamma)= \mu(s(\gamma)). $$
We can set on $U$: 
\begin{equation}
  \nP(\gamma) = \rmP(\gamma) . \mu^{1/2}_{s}\mu^{1/2}_{r}
\end{equation}
where $\rmP$ is a scalar oscillatory integral conveniently given  in the following form:
\begin{equation}\label{eq:P-con-dist}
\rmP(\gamma) = \int e^{-i\gamma^{-1}.\xi'} p_{0}(r(\gamma),\xi')d\xi'. 
 \end{equation}
Let us describe the various ingredients of this formula.  First, $p_{0}\in  S^{m}(A^*G)$ is a (classical) symbol, and the integral (in the distribution sense) is performed with respect to $(0,\xi')\in A_{r(\gamma)}^{*}G \subset T_{r(\gamma)}^{*}G$. 
 Secondly, it is understood that    $\gamma^{-1}=\iota(\gamma)$  stands for the $n$-tuple of coordinates of the inverse of $\gamma$ in $G$, and then  $\gamma^{-1}.\xi'$ stands for its scalar product with $(0,\xi')$  in $\RR^{n}$. We could use the inverse of an exponential map to give an invariant meaning to $\gamma^{-1}.\xi$ with $\xi \in A_{r(\gamma)}^{*}G$, but since we already work in local coordinates, this is pointless. 
Finally, we read from \eqref{eq:P-con-dist} that: 
\begin{equation}
 \sigma_0(\nP) = p_{0} \mod S^{m-1}(A^*G),
\end{equation}
and since the symbols used here are classical, we may identify  $\sigma_0(\nP)$ with the leading homogeneous part $p_{0}^{0}$ of $p_{0}$.  
Now let  $u \in C^{\infty}_{c}(G)$ with support in a local chart  $V$ of $G$, and set:
 \begin{equation}
 u(\gamma) = \rmu(\gamma) .  \mu^{1/2}_{s}\mu^{1/2}_{r}
\end{equation}
with $\rmu\in C^{\infty}_{c}(V, \CC)$. To express $\nP {u}$ in local coordinates in terms of $\rmP$ and $\rmu$, we need to recall the necessary identifications of densities allowing the convolution product:
 \begin{equation}\label{eq:abstract-convolution-P-A}
 \nP({u})(\gamma)= \int_{\alpha \in G_{s(\gamma)}} \nP(\gamma\alpha^{-1}){u}(\alpha).
\end{equation}
For that purpose, note that for any $\gamma,\alpha$ with same source point:
$$ \mu_{s}(\gamma\alpha^{-1})=\mu(r(\gamma))= \mu_{s}(\gamma), \   \mu_{r}(\gamma\alpha^{-1})=\mu(r(\alpha))= \mu_{s}(\alpha), \ \mu_{r}(\alpha)= \mu_{r}(\gamma).  $$ 
Hence:
$$  \nP(\gamma\alpha^{-1}){u}(\alpha) = \rmP(\gamma\alpha^{-1})\rmu(\alpha) \mu_{s}(\alpha)   \mu^{1/2}_{s}(\gamma)\mu^{1/2}_{r}(\gamma). $$
It remains to express $ \mu_{s}(\alpha) $ in term of a one density on $G_{x}\cap V$.  
We also assume that the coordinates fixed on $V$ trivializes the source map $s$: 
$$V\ni \alpha = (\alpha'',\alpha') \text{ with } s(\alpha) = \alpha''$$
In the coordinates fixed on $U$ and $V$, we get using $ (dR_{\alpha})_{r(\alpha)} : (r^{*}AG)_{\alpha} \overset{\simeq}{\longrightarrow} (\ker ds )_{\alpha}$:
$$\mu_{s}(\alpha)= | dx'| = |(dR_{\alpha})_{r(\alpha)}|^{-1}| d\alpha'|. $$
It follows that,  setting $\widetilde{\nP}({u}) = \rmv \mu^{1/2}_{s}\mu_{r}^{1/2}$ on $W$:
\begin{align}
\rmv(\gamma) & = \int P(\gamma\alpha^{-1})\rmu(\alpha)  |(dR_{\alpha})_{r(\alpha)}|^{-1} d\alpha' \\
& = \int e^{-i\alpha\gamma^{-1}.\xi'} p_{0}(r(\gamma),\xi') |(dR_{\alpha})_{r(\alpha)}|^{-1} \rmu(\alpha) d\alpha' d\xi'.
\end{align}
Actually, the action of the induced family of operators $\nP_{x}\in \Psi^{m}(G_{x}, \Omega^{1/2}_{G_{x}})$ on half-densities ${f}\in C^{\infty}_{c}(G_{x},\Omega^{1/2}_{G_{x}})$ is given by the same formula: 
\begin{equation}\label{pdo-algebraic-style}
\text{ if }{f} = \rmf \mu_{s}^{1/2}, \text{ then }\nP_{x}({f}) = \rmv \mu_{s}^{1/2} \text{ with }  \rmv(\gamma) = \int e^{-i\alpha\gamma^{-1}.\xi} p_{0}(r(\gamma),\xi) \rmf(\alpha) |(dR_{\alpha})_{r(\alpha)}|^{-1} d\alpha'd\xi.
\end{equation}
Let us set 
\begin{equation}
\varphi(\alpha)  = |(dR_{\alpha})_{r(\alpha)}|^{-1}.
\end{equation}
Since $\alpha\gamma^{-1}$ vanishes at $\alpha=\gamma\in G_{s(\gamma)}$, there exists a linear map 
\begin{equation}
 \psi(\alpha,\gamma)  : \RR^{n^{(1)}} \longrightarrow  \RR^{n^{(1)}}
\end{equation}
which is $C^\infty$ in $(\alpha,\gamma)$, bijective for $\alpha$ in a neighborhood of $\gamma$ and satisfies: 
\begin{equation}
    \alpha\gamma^{-1}  = \psi(\alpha,\gamma)(\alpha-\gamma). 
\end{equation}
By construction we have:
\begin{equation}
\psi(\gamma,\gamma)= (dR_{\gamma^{-1}})_{\gamma} = (dR_{\gamma})_{r(\gamma)}^{-1}.
\end{equation}
Now we work on \eqref{pdo-algebraic-style} to find the amplitude and symbol of $\nP_{x}$ in local coordinates on $V$:
\begin{align}\label{pdo-friendly-style}
\rmv(\gamma) & = \int e^{-i\langle \alpha-\gamma,{}^{t}\psi(\alpha,\gamma)\xi' \rangle} p_{0}(r(\gamma),\xi') \rmu(\alpha)\varphi(\alpha) d\alpha'd\xi'  \nonumber\\
 & = \int e^{i\langle \gamma - \alpha,\xi' \rangle} \big(p_{0}(r(\gamma),{}^{t}\psi(\alpha,\gamma)^{-1}\xi')\varphi(\alpha)  \vert {}^{t}\psi(\alpha,\gamma)\vert ^{-1}\ \rmu(\alpha) d\alpha'd\xi'  \nonumber\\
 & =  \int e^{i\langle \gamma - \alpha,\xi' \rangle} \widetilde{p}(\gamma,\alpha,\xi') \rmu(\alpha) d\alpha'd\xi' \nonumber\\
 & =  \int e^{i\langle \gamma - \alpha,\xi' \rangle} p(\gamma,\xi') \rmu(\alpha) d\alpha'd\xi'
\end{align}
where we have set 
\begin{equation}\label{eq:ampli-Px}
 \widetilde{p}(\gamma,\alpha,\xi')  = p_{0}(r(\gamma),{}^{t}\psi(\alpha,\gamma)^{-1}\xi')\varphi(\alpha)  \vert {}^{t}\psi(\alpha,\gamma)\vert ^{-1}
\end{equation}
and 
\begin{equation}\label{eq:symb-total-Px}
 p(\gamma,\xi') = e^{i\langle D_{\alpha'},D_{\xi'} \rangle}\widetilde{p}(\gamma,\alpha,\xi')\vert_{\alpha=\gamma}
\end{equation}
which gives the asymptotic expansion:
\begin{equation}\label{eq:symb-asympt-exp-Px}
 p(\gamma,\xi') \sim \sum \frac1{k!}\langle iD_{\alpha'},D_{\xi'} \rangle^{k}\widetilde{p}(\gamma,\alpha,\xi')\, \vert_{\alpha=\gamma}.
\end{equation}
Since  $(r(\gamma), {}^t\psi(\gamma,\gamma)^{-1}\xi') = (r(\gamma), {}^t (dR_{\gamma})_{{}_{r(\gamma)}}\xi') = \rg(\gamma,\xi')$,  the expression of the principal symbol of $\nP_{x}$ over $V$ is the first term in the sum \eqref{eq:symb-asympt-exp-Px}:
\begin{equation}
 \sigma(\nP_{x})(\gamma,\xi') = p_{0}(\rg(\gamma, \xi'))  \mod S^{m-1}(T^*_sG),
\end{equation}
or equivalently using  homogeneous expansions:  $\sigma(\nP) =p^{0} = p_{0}^{0}\circ \rg$.   
 \end{proof}

\begin{rmk}\label{rmk:convention-extension}
We will often consider $C^{\infty}$ functions on $T^*_sG$ as $C^{\infty}$ functions on $T^{*}G$, thanks to the convention $a(\gamma,\xi)=a(\gamma,\xi\vert_{T_{s(\gamma)}G})$. 
\end{rmk}
 
We now turn our attention to the sub-principal symbols. It is not obvious to us how to define the sub-principal symbol for general conormal distributions, but in the case of  $\Psi^{*}_{G}=I(G,M,\Omega^{1/2})$, we  may again consider the family of usual sub-principal symbols of the operators $\nP_x\in \Psi^{m}(G_{x}, \Omega^{1/2}_{G_{x}})$ and set: 
 \begin{equation}
 (\gamma,\xi)\in T_s^*G,\quad  \sigma^{1s}(\nP)(\gamma,\xi) := \sigma^{1s}(\nP_{s(\gamma)})(\gamma,\xi) \in S^{[m-1]}(T_s^*G).
\end{equation}
 This gives a well defined map: 
 \begin{equation}
   \sigma^{1s} : \Psi^m_{G} \longrightarrow S^{[m-1]}(T_s^*G),
\end{equation}
When $\nP$ is given by \eqref{eq:P-con-dist},  we recall that the sub-principal symbol above is given in terms of the homogeneous expansion of total symbol  $p$, expressed in the last proof (see formula \eqref{eq:ampli-Px}, \eqref{eq:symb-total-Px} and \eqref{eq:symb-asympt-exp-Px}), by:
\begin{equation}\label{eq:repres-symb-sous-princip}
  p^{1s}(\gamma, \xi') := \sigma^{1s}(\nP)(\gamma,\xi') = p^{1}(\gamma,\xi')  - \frac{i}2 \langle D_{\gamma'},D_{\xi'} \rangle p^{0}(\gamma,\xi'),
\end{equation}
\bigskip We now consider  $ p^{0} = \sigma(\nP)$ as a  $C^\infty$ homogeneous function on $ T^{*}_{\mbt} G = T^{*}G\setminus \ker \rg $ (see Remark \ref{rmk:convention-extension}) and we denote  
 $H_{p^{0}}\in \Gamma(TT^{*}_{\mbt} G)$ the hamiltonian vector field of $p^{0}$. We recall that the latter is defined by $dp^{0}(\cdot)=\omega_G(H_{p^{0}}, \cdot)$,  and  in local coordinates $(\gamma, \xi)$ we get: 
$$
  \quad  H_{p^{0}}=\sum_{j=1}^n \frac{\partial p^{0}}{\partial \xi_j}\frac{\partial }{\partial \gamma_j}-\frac{\partial p^{0}}{\partial \gamma_j}\frac{\partial }{\partial \xi_j}.
$$ 
Now we shall compute the principal symbol of a product $\nP\nA$ where $\nP$ is a $G$-PDO and $\nA$ a $G$-FIO in the situation later encountered in the construction  of the parametrix of $e^{it\nP}$. To that purpose, we recall that the principal symbol of $G$-FIO is a homomorphism \cite{Horm-classics}:
\begin{equation}\label{eq:symbol-fio-brut}
I^{m}(G,\Lambda; \Omega^{1/2}) \To S^{[m+n/4]}(\Lambda,M_\Lambda\otimes\Omega^{1/2}_{\Lambda}\otimes  \hat{\Omega}^{1/2} \otimes \hat{\Omega}^{-1/2}_{G})
\end{equation}
where $M_{\Lambda}$ is the Maslov bundle and $\hat{E}$ denotes the pull back of the vector bundle $E\to G$ over $\Lambda$. By \cite{LV1}, we  know that there is canonical isomorphism:
\begin{equation}\label{eq:iso-bizar-dens-nice-dens}
 \hat{\Omega}^{1/2} \otimes \hat{\Omega}^{-1/2}_{G} \simeq  \Omega^{1/2}_{\rg} = \sg^{*}\Omega^{1/2}(AT^{*}G)
\end{equation}
This isomorphism  uses the product and inversion map of $G$ but their contributions cancel and thus, elements in $ \hat{\Omega}^{1/2} \otimes \hat{\Omega}^{-1/2}_{G}$ do define, without any other data, pull back of half densities on the vector bundle $AT^{*}G\To A^{*}G$. We thus may  consider the principal symbol of Fourier integral $G$-operators as a homomorphism: 
\begin{equation}\label{eq:symbol-fio-fine}
I^{m}(G,\Lambda;\Omega^{1/2}) \To S^{[m+n/4]}(\Lambda,M_\Lambda\otimes\Omega^{1/2}_{\Lambda}\otimes \Omega^{1/2}_{\rg})
\end{equation}
We recall that for a manifold $X$ and a vector field $V$ on $X$ with flow $\phi_{t}$, the Lie derivative of a $\alpha$-density $a$ is the $\alpha$-density given by, in local coordinates $a = \mathsf{a} \vert dx\vert^{\alpha}$: 
\begin{equation}\label{eq:lie-deriv-alpha-dens-Lambda}
 \cL_{V}(\mathsf{a} \vert dx\vert^{\alpha}) =  \frac{d}{dt}\phi^{*}_{t}\mathsf{a} \vert dx\vert^{\alpha}\vert_{t=0} = \big(V\cdot \mathsf{a} + \alpha\, \mathrm{div}(V)\mathsf{a}\big)  \vert dx\vert^{\alpha} .
\end{equation}
This is the same for sections $a\in C^{\infty}(\Lambda, M_{\Lambda}\otimes \Omega^{1/2}_{\Lambda})$ and vector fields $V\in \Gamma(T\Lambda)$. Indeed, the transition functions of  $M_{\Lambda}$ are locally constant, so the bundle $M_{\Lambda}$ can be factorized out of \eqref{eq:lie-deriv-alpha-dens-Lambda}. 

On the other hand, we are mainly interested in Hamiltonian vector fields $V = H_{f}$ that are also right invariant, which happens if and only if  $f = f_{0}\circ \rg$ \cite{CDW}, and such that $f\vert_{\Lambda} = 0$, which implies that $V$ is tangent to $\Lambda$. 
For such vector fields, we can extend the Lie derivative above to a map $\cL_{f} =\cL_{H_{f}}$ acting on sections of the line bundle appearing in the symbols  space in \eqref{eq:symbol-fio-fine}.
To do that,  consider  $\nu_{\rg} = \nu\circ\sg\in C^{\infty}(T^{*}G, \Omega^{1/2}_{\rg})$ with $\nu$ a positive density on $AT^{*}G$. Since by assumption 
$s_{\sg}\circ\phi_{t} = \sg$, we get:
\begin{align*}
 \cL_{f}(\nu_{\rg})  : =  \frac{d}{dt}\phi^{*}_{t}(\nu_{\rg})\vert_{t=0} 
    =  \frac{d}{dt}\nu(\sg\circ\phi_{t})\vert_{t=0}  =  \frac{d}{dt}\nu^{1/2}(\sg)\vert_{t=0}  = 0.
  \end{align*}
Combining the usual action of $V\vert_{\Lambda}$ recalled in \eqref{eq:lie-deriv-alpha-dens-Lambda} with the above trivial one,  we obtain that $V=H_{f}$ acts on $C^{\infty}(\Lambda, M_\Lambda\otimes\Omega^{1/2}_{\Lambda}\otimes\Omega^{1/2}_{\rg})$ by the formula:
\begin{align}\label{eq:yes}
 \cL_{f}(\mathsf{a}\mu_{\Lambda}^{1/2}\nu_{\rg}^{1/2}) =  \big(H_{f}\cdot \mathsf{a} + \frac12\, \mathrm{div}(H_{f})\mathsf{a}\big)\mu_{\Lambda}^{1/2}\nu_{\rg}^{1/2}.
  \end{align}
In the important particular case where the $G$-relation $\Lambda$ is a bissection, that is, when $\sg,\rg$ are diffeomorphisms from $\Lambda$ to open subsets of $A^{*}G$, then 
\begin{equation}
T\Lambda \oplus \ker d\sg\vert_{\Lambda} = T_{\Lambda}T^{*}G. 
\end{equation}
Since $v = H_{f_{0}\circ \rg}$ is tangent to both $\Lambda$ and to the $\sg$-fibers, we conclude that $v$ vanishes on $\Lambda$, which implies that $\cL_{f} = 0$ in \eqref{eq:yes}.

\begin{thm}\label{prop:adapt-25-2-4}
Let $\Lambda$ be a $G$-relation  and  $Q\in\Psi_{G,c}^m$ with principal and sub-principal $G$-symbols $q^{0}$ and $q^{1s}$.  Assume  that  $q^{0}$ vanishes on $\Lambda$. Let $\nA\in I^{m'}(G,\Lambda;\Omega^{1/2})$ and let $a\in S^{m'+n/4}(\Lambda,M_\Lambda\otimes\Omega^{1/2}_{\Lambda}\otimes  \hat{\Omega}^{1/2}\otimes \hat{\Omega}^{-1/2}_{G})$ be a principal symbol of $\nA$. \\
Then 
\begin{equation}
   Q\nA \in I^{m+m'-1}(G,\Lambda;\Omega^{1/2})
\end{equation}
and $\nP\nA$ has a principal symbol represented by
\begin{equation}
 -i\cL_{q^{0}}a+ q^{1s}a.  
\end{equation}
\end{thm}
We could consider the distribution $Q\nA$ as the family of operators $Q_{x}\circ \nA_{x}$ and apply \cite[Theorem 25.2.4]{Horm-classics}. However, we are going to consider $Q\nA$ as a single lagrangian distribution  on $G$ given by the convolution in $G$ of two distributions, and  then make the minor necessary adaptations of  the proof of \cite[Theorem 25.2.4]{Horm-classics}. This yields  more conceptual and self-contained explanations for the assertion to be proved.  
\begin{proof}
We keep the assumptions and notation introduced for $Q$ in the proof of Proposition \ref{prop:princ-symbol-2pov}. 
Using a partition of unity and \cite{LV1}, we  can assume that $\nA$ is supported in the domain $V$ of local coordinates trivialising $s$ such that there exists a  conic open set $C$ in $\RR^{n}$ and a  homogeneous $C^\infty$ function $h$ such that:
\begin{equation}
   \Lambda\cap T^{*} V= \{ (h'(\xi),\xi)\ ;\  \xi\in C\} .
\end{equation}
The existence of such coordinates follows from \cite[Lemma 25.2.5 and Theorem 21.2.16]{Horm-classics}.
We can write in these local coordinates above: 
\begin{equation}\label{eq:loc-exp-A}
   \nA = \rmA.  \mu^{1/2}_{s}\mu^{1/2}_{r} \text{ with } \rmA(\gamma)=  \int e^{i(<\gamma,\eta> - h(\eta))} \rma(\eta)d\eta,
\end{equation}
where $\mathsf{a}\in S^{m' - n/4}(\RR^{n})$ has support in a conic neighborhood of $C$.
Then on $V$: 
\begin{equation}\label{eq:complicated-PA}
Q\nA= \rmB \mu^{1/2}_{s}\mu^{1/2}_{r} \text{ with } \rmB(\gamma) =  \int e^{i (-<\alpha\gamma^{-1},\xi> + <\alpha,\eta>-h(\eta))} q_{0}(r(\gamma),\xi)\mathsf{a}(\eta)\varphi(\alpha) d\alpha d\xi d\eta .
\end{equation}
Remember that, according to the decomposition  $\eta=(\eta'',\eta')$ provided by the  local trivialisation of $s$,  the symbol  $q(\gamma,\eta')$ of $Q_{s(\gamma)}$ (see \eqref{pdo-friendly-style} and above) is given by
\begin{equation}\label{eq:class-formula-symbol}
q(\gamma,\eta')= e^{-i\langle \gamma,\eta' \rangle}Q_{s(\gamma)}(e^{i\langle \alpha,\eta' \rangle}) = \int e^{i (-<\alpha\gamma^{-1},\xi>+<\alpha-\gamma,\eta'>)}q_{0}(r(\gamma),\xi)\varphi(\alpha)d\alpha d\xi.
\end{equation}
Since $\gamma,\alpha\in G_{s(\gamma)}$, the same identity is licit for $q$ considered as a function of $(\gamma,\eta)$, but does not define anymore a symbol in general (it satisfies symbolic estimates of order $m$ in $\eta'$ but is independent of $\eta''$). However, the assumption on the wave front set of $\nA$ implies that the symbol $a$ is or order $-\infty$ in some open  cone around $(\eta'',0)$. Indeed, $(\eta'',0) \in \ker(s_\Gamma)$ and by assumption $\Lambda$ is a $G$ relation, hence $\WF{A} \cap\ker(s_\Gamma)= \emptyset$. Therefore, the product $\rmb(\gamma,\eta) = q(\gamma,\eta)\mathsf{a}(\eta)$ is a symbol of order $m+m'-n/4$. We get from \eqref{eq:complicated-PA} and \eqref{eq:class-formula-symbol}:
\begin{equation}\label{eq:concise-developped-PA}
 \rmB(\gamma)= \int e^{i ( <\gamma,\eta> - h(\eta))}  q(\gamma,\eta)\mathsf{a}(\eta) d\eta,
\end{equation}
It is a lagrangian distribution of order $m+m'$ a priori, but since the leading part of $\rmb$, which is represented  $\rmb^0(\gamma,\eta)=q^{0}(\gamma,\eta)\mathsf{a}(\eta)$, vanishes on $\Lambda$, it is actually of  order $m+m'-1$ and we need to work out more the expression \eqref{eq:concise-developped-PA} to get its principal symbol.  To this end, we set: 
\begin{equation}
q(\gamma,\xi)=q^{0}(\gamma,\xi) + e(\gamma,\xi),
\end{equation}
and using  the assumption: $q^{0}(\gamma,\xi)= 0$  whenever $\gamma=h'(\xi)$; we make the factorisation:
\begin{equation}\label{eq:factorize-p0}
q^{0}(\gamma,\xi)= \sum_j q_j(\gamma,\xi)(\gamma_j- \frac{\partial h}{\partial \xi_j}).
\end{equation}
Now, after an integration by parts in \eqref{eq:concise-developped-PA}, we get:
\begin{equation}
 \rmB(\gamma)= \int e^{i (<\gamma,\xi>-h(\xi))} (e\mathsf{a} - \sum_j D_{\xi_j}(q_j \mathsf{a}))d\xi. 
\end{equation}
It follows that $Q\nA$ has principal symbol represented by 
\begin{equation}
(e\mathsf{a} - \sum_j D_{\xi_j}(q_j \mathsf{a} ))(h'(\xi),\xi) \vert d\xi\vert^{1/2} (\mu_{r}\mu_{s}\vert d\gamma\vert^{-1})^{1/2}
\end{equation}
while $\nA$ has principal symbol represented by 
\begin{equation}
 \mathsf{a}(\xi) \vert d\xi\vert^{1/2} (\mu_{r}\mu_{s}\vert d\gamma\vert^{-1})^{1/2}.
\end{equation}
Since $H_{q^{0}}$ is tangent to $\Lambda$ we have on $\Lambda$ parametrized by $\xi$:
 \begin{equation}
     H_{q^{0}}(\xi) = - \sum_{j}\frac{\partial q^{0}}{\partial \gamma_j}(h'(\xi),\xi)\frac{\partial }{\partial \xi_j} = -\sum_{j} q_{j}(h'(\xi),\xi)\frac{\partial }{\partial \xi_j}. 
 \end{equation}
 Then, as a vector field on $\Lambda$, the divergence of $H_{q^{0}}\in\Gamma(T\Lambda)$ is given by:
\begin{equation}\label{eq:tr-hess-p}
\mathrm{div}(H_{q^{0}}) =  - \sum_{j} \frac{\partial}{\partial \xi_j}\Big\lbrack q_j(h'(\xi),\xi)\Big\rbrack= 
   - \sum_{j,k}\frac{\partial^2 h}{\partial \xi_k\partial\xi_j}\frac{\partial q_j}{\partial \gamma_k } - \sum_j \frac{\partial q_j}{\partial \xi_j}.
\end{equation}
On the other hand we have:
\begin{align*}
\frac{\partial^2 {q^{0}}}{\partial \xi_k\partial\gamma_k}(\gamma,\xi) &= 
\frac{\partial}{\partial \xi_k}\sum_j (\frac{\partial}{\partial\gamma_k} q_j).(\gamma_j-h'_j(\xi)) + q_k \\
& =  \sum_j \frac{\partial^2 q_j}{\partial\gamma_k\partial \xi_k}.(\gamma_j-h'_j(\xi)) -\sum_j \frac{\partial q_j}{\partial\gamma_k}\frac{\partial^2 h}{\partial\xi_k\partial\xi_j} + \frac{\partial q_k}{\partial \xi_k},
\end{align*}
which gives after evaluation at $\gamma=h'(\xi)$:
\begin{equation}
\sum_{k}\frac{\partial^2 {q^{0}}}{\partial \xi_k\partial\gamma_k}(h'(\xi),\xi)= \sum_{k}\frac{\partial q_k}{\partial \xi_k}-\sum_{j,k} \frac{\partial q_j}{\partial\gamma_k}\frac{\partial^2 h}{\partial\xi_k\partial\xi_j} =  2\sum_{k}\frac{\partial q_k}{\partial \xi_k}  + \mathrm{div}(H_{q^{0}}).
\end{equation}
Setting 
$$ \mathsf{L}_{q^{0}}(a) :=\cL_{q^{0}}(a)  (\mu_{r}\mu_{s}/\vert d\gamma\vert)^{-1/2} =   - \sum_j q_j\frac{\partial \mathsf{a}}{\partial \xi_j}  + \frac12 \mathrm{div}(H_{q^{0}})\mathsf{a}, $$
we get, still for $\gamma=h'(\xi)$:
\begin{align}\label{al:extract-symbol}
(e\mathsf{a} - \sum_j D_{\xi_j}(q_j \mathsf{a}))  & = \Big(e + i \sum_j \frac{\partial q_{j}}{\partial \xi_{j}}\Big)\mathsf{a} + i\sum_j q_j\frac{\partial \mathsf{a}}{\partial \xi_j} \notag \\
 & = \Big(e + i \sum_j \frac{\partial q_{j}}{\partial \xi_{j}}   + \frac{i}{2}\mathrm{div}(H_{q^{0}})\Big)\mathsf{a} - i\mathsf{L}_{q^{0}}(a) \notag \\
 & = \Big(e  + \frac{i}{2}\sum_{k} \frac{\partial^2 {q^{0}}}{\partial \xi_k\partial\gamma_k} \Big)\mathsf{a}  -  i\mathsf{L}_{q^{0}}(a) \notag \\
 &=  - i\mathsf{L}_{q^{0}}(a) + \left(e + \frac{i}{2}\sum_k  \frac{\partial^2 {q^{0}}}{\partial \xi'_k\partial\gamma'_k} \right)\mathsf{a}. 
\end{align}
In the last line, we have decomposed $\xi=(\xi',\xi'')$ where $\xi'$ is cotangent to the $s$-fibers and used the fact that $q^{0}$ does not depend on the $\xi''$ variables. This proves, looking at formula \eqref{eq:repres-symb-sous-princip}, that  $ - i\cL_{q^{0}}(a) +  q^{1s} . a $ is a principal symbol of $Q\nA$.
\end{proof}

\section{The  Hamiltonian flow of the principal symbol and the associated $G$-relations}\label{sec:evo-lag-sub}
We set: 
 \[
  \Lambda_0 =  A^*G\setminus 0,  \quad  T^{*}_{\mbt} G = T^{*}G\setminus \ker \rg \ \text{ and } \quad \dt{T}^{*}G = T^{*}G\setminus (\ker \rg\cup\ker \sg). 
 \]
The operator $P$ being elliptic, we have $\WF{P}= \Lambda_0$. Recall that $p^{0}_0\in C^\infty(\Lambda_0)$ denotes the homogeneous representative of $\sigma_0(P)$ and that $ p^{0} = p^{0}_0\circ \rg$ is then the homogeneous representative of $\sigma(P)\in C^{\infty}(T^{*}_{s}G\setminus 0)$. 

For every $a\in\RR$ and $(\gamma,\xi)\in T^*G$, we let $a.(\gamma,\xi) =\Dil_a(\gamma,\xi) = (\gamma,a\xi)$. 
\begin{prop}\label{prop:hamilton-flow} The flow $\chi$   of $H_{p^{0}}$ satisfies the following properties:
\begin{enumerate}
\item It is complete. 
 \item  It commutes with dilations in $T^*G$:  
 $$ 
   \forall a\in \RR^*_+,\ t \in \RR,\  (\gamma, \xi)\in T^{*}_{\mbt} G,\quad   a.\chi(t,\gamma, \xi)= \chi(t,\gamma,a\xi). 
 $$ 
 \item  It provides at each time $t$ a section of $s_\Gamma$ and commutes with right multiplication: 
 $$ 
 \forall t \in \RR, \ (\delta_1,\delta_2)\in \Gamma^{(2)},\quad \sg(\chi(t,\delta_1))=\sg(\delta_1) \text{ and } \chi(t,\delta_1\delta_2) = \chi(t,\delta_1)\delta_2.
 $$ 
In particular, the integral curves of $H_{p^{0}}$ go along the fibers of $s_\Gamma$. 
\end{enumerate}
\end{prop}
\begin{proof} This is essentially contained in \cite{CDW}. More precisely:
\begin{enumerate}
 \item[(2)] The homogeneity of  $p^{0}$ implies: 
 $(\Dil_a)_*(H_{p^{0}}) = H_{p^{0}}$ and therefore  $\Dil_a\circ \chi_{t}\circ \Dil_a^{-1} = \chi_{t}$, which gives the result. 

\item[(3)] By definition, we have $\omega(H_{p^{0}},X)=dp^{0}(X)=dp_{0}^{0}(d\rg(X))$ which yields $H_{p^{0}}\in (d\ker \rg)^\omega=d\ker \sg$.   Using the last part of the proof of \cite[Lemma, p.22]{CDW}, we get that $H_{p^{0}}$ is a right invariant vector field, which proves that the flow goes along the $s$-fibers and is right invariant. 
\item[(1)] Now, by compacity of $M=G^{(0)}$  and the homogeneity of $H_{p^{0}}$ in the fibers of $\Lambda_0$, we get the existence of $\epsilon>0$ such that $\chi : ]-\epsilon,\epsilon[\times \Lambda_0\to T^{*}_{\mbt} G$ is well defined. By right invariance, we extend $\chi_p$ onto $]-\epsilon,\epsilon[\times (T^{*}_{\mbt} G)$ and using the group property of flows we can choose for any $t$ an integer $N$ such that $|t|/N< \epsilon$ and set 
\[
 \chi(t,\alpha)=  \chi_{t/N}\circ\ldots \circ\chi_{t/N}(\alpha)
\]
which proves completeness of the flow. 
\end{enumerate}
\end{proof}
\begin{rmk}\label{rmk:prop-flow} \ 
 \begin{enumerate}
  \item The compacity of $M=G^{(0)}$ is only needed here for the completeness of $H_{p^{0}}$.
  \item By construction, $\chi_{t}$ is a diffeomorphism of $T^{*}_{\mbt} G$  and since $s_\Gamma\circ \chi_t = s_\Gamma$, we have:
  \begin{equation}\label{eq:flow_preserves_adm_part}
   \chi_{t}(\dt{T}^{*}G)= \dt{T}^{*}G.
  \end{equation}
 \end{enumerate}
\end{rmk}
We now set: 
\begin{equation}
 \forall t\in \RR,\quad \Lambda_t = \chi_{t}(\Lambda_0) \subset \dt{T}^{*}G.
\end{equation}
\begin{prop}\label{prop:time-lag}
 For any real number $t$, the set $\Lambda_t$ is an invertible $G$-relation and:
 \begin{equation}
  \forall t_1,t_2\in\RR,\quad \Lambda_{t_1}.\Lambda_{t_2} = \Lambda_{t_1+t_2}.
 \end{equation}
\end{prop}
\begin{proof}
Since $\chi$ is the Hamilton flow of a homogeneous function, each $\chi_t$ is a homogeneous symplectomorphism. Since $\Lambda_0$ is  a homogeneous Lagrangian, its image $\Lambda_t$ by  $\chi_t$ is then a Lagrangian homogeneous submanifold, contained by construction in $\dt{T}^{*}G$. Thus, $\Lambda_t$ is a $G$-relation. Since $s_\Gamma\circ \chi_t|_{\Lambda_0} = \Id_{\Lambda_0}$, we get that $s_\Gamma|_{\Lambda_t}$ is a diffeomorphism. The same conclusion is true for $r_\Gamma|_{\Lambda_t}$ because the vector field $H_{p^{0}}$ is right invariant  and therefore $r_\Gamma\circ \chi|_{\RR\times \Lambda_0}$ is the (complete) flow of the vector field $(r_\Gamma)_*(H_{p^{0}})\in C^\infty(\Lambda_0,T\Lambda_0)$ defined by:     $(r_\Gamma)_*(H_{p^{0}})(\delta) = (dr_\Gamma)_\delta(H_{p^{0}}(\delta)), \delta\in\Lambda_0$. That $s_\Gamma|_{\Lambda_t}$ and $r_\Gamma|_{\Lambda_t}$
are diffeomorphisms mean precisely that $\Lambda_t$ is an invertible $G$-relation  \cite{LV1}  (or a lagrangian bissection, if one accepts as bissections submanifolds of $\Gamma$ onto which $r_\Gamma$ and $s_\Gamma$ are diffeomorphisms onto open subsets of $A^{*}G = (T^{*}G)^{(0)}$). 

Let us proceed to the proof of the one parameter group relation.  Let $\delta_j=\chi_{t_j}(u_j)\in \Lambda_{t_j}$, $j=1,2$, be two composable elements. Then $u_1 = s_\Gamma(\delta_1) = r_\Gamma(\delta_2)$ and by commutation of $\chi$ with right multiplication:
$$
  \delta_1.\delta_2 = \chi_{t_1}(r_\Gamma(\delta_2))\delta_2 =  \chi_{t_1}( \delta_2) = \chi_{t_1}( \chi_{t_2}(u_2))=\chi_{t_1+t_2}(u_2)\in  \Lambda_{t_1+t_2}. 
$$ 
The converse inclusion is then obvious. 
\end{proof}

\section{Global aspects of the family $(\Lambda_t)_t$}\label{sec:global-Lambda}
In the vocabulary of \cite{LV1}, the family $(\Lambda_t)_{t\in\RR}$ admits a gluing into a single Lagrangian submanifold $\Lambda\subset T^*(\RR\times G)$. The expression of $\Lambda$ is  actually straightforward and we shall study it in relation with the both groupoid  structures on $\RR\times G$.   
\begin{prop} Let $i_t :G \to \RR\times G$ be the inclusion $\gamma \mapsto(t, \gamma)$. The set
\begin{equation} \label{eq:Lambda}
 \Lambda = \left\{ (t, -p^{0}(\chi_t(x,\xi)), \chi_t(x,\xi) ) \in T^*( \RR\times G)\ ;  \ t \in \RR,  (x,\xi)\in A^*G\setminus 0   \right\}
\end{equation}
is a  conic Lagrangian submanifold of $T^*( \RR\times G)$ satisfying:
\[
   \forall t\in\RR, \quad i_t^*\Lambda = \Lambda_t. 
\]
\end{prop}
\begin{proof}  the map $F(t,\delta)=(t,\chi_t(\delta))$ being a diffeomorphism of $\RR\times (T^{*}_{\mbt} G)$, the set  $F(\RR\times\Lambda_0)$  is a submanifold of  $\RR\times T^*G$ and therefore $\Lambda$, as a graph,  is a submanifold of $T^*( \RR\times G)$, obviously homogeneous and  $\phi(t,x,\xi) =  (t, -p^{0}(\chi_t(x,\xi)), \chi_t(x,\xi) )$ is a parametrization. 
   
We check that $\Lambda$ is lagrangian, which is equivalent by homogeneity of $\Lambda$ to the vanishing on it of the canonical one form $\alpha_{\RR\times G} = \tau dt + \alpha_G$, that is to the vanishing of the  one form on $\RR\times\Lambda_0$ defined by  $\phi^{*}\alpha_{\RR\times G}$.  We have:
\begin{align*}
\phi^{*}\alpha_{\RR\times G} & = -p^{0}(\chi) dt + \chi^{*}\alpha_G \\
 & =  -p^{0}(\chi) dt + \alpha_G( \frac{\partial \chi}{\partial t})dt +  (\chi_t)^*\alpha_G
\end{align*}
Since   $\chi_t$ is a homogeneous symplectomorphism of $T^{*}_{\mbt} G$, the one form $(\chi_t)^*\alpha_G$ vanishes on $\Lambda_0$. On the other hand, by homogeneity of $ p^0$ and Euler formula:
\begin{equation}
 \alpha_G( \frac{\partial \chi}{\partial t})  = \alpha_G(H_{p^{0}}(\chi_t)) = \sum_j \xi_j \frac{\partial p^{0}}{\partial \xi_j}(\chi)  = p^0(\chi) ,
\end{equation}
which proves the required assertion. 
 \end{proof}
There are  two natural structures of groupoid on $\RR\times G$, with different unit space:
 \[
  \widetilde{G} = \RR\times G \rightrightarrows \RR\times G^{(0)} \text{ and } G_+ = \RR\times G \rightrightarrows   G^{(0)} 
 \]
The first is the (constant) family of groupoids $G_t=G$ parametrized by the space $\RR$, and the second one is the cartesian product of $G$ with the additive group $\RR$. 
The corresponding symplectic groupoid structures on $T^*(\RR\times G)$ will be denoted by: 
\begin{equation}
 \widetilde{\Gamma} = T^*(\RR\times G) \rightrightarrows  \pr_2^* A^*G  \ \text{ and }\  \Gamma_+ =  T^*(\RR\times G) \rightrightarrows   \RR\times A^*G
\end{equation}
where $\pr_2 : \RR\times M\to M$ is the second projection and $\RR\times A^*G$ denotes the bundle over $M$ with fibers $\RR\times A_x^*G$.

We will say that a subset $A\subset \RR\times X$ is $\RR$-proper if $\pr_{1} : A \to \RR$ is proper, that is 
$$ A \cap [a,b]\times X \text{ is compact for any } a, b\in \RR. $$
We will call support of $A\subset E$ the set $\supp A = \pi(A)\subset X$ for any bundle map $\pi : E \to X$. 
\begin{prop}\label{prop:more-on-Lambda}
The submanifold $\Lambda$ of $T^*(\RR\times G)$ satisfies the following:
\begin{enumerate}
\item It is contained in $(T^*\RR\setminus 0)\times \dt{T}^{*}G$ and closed in $T^*\RR\times(T^{*}_{\mbt} G)$.
\item It is both an invertible $\widetilde{G}$-relation and a family $G_+$-relation.
\item The support of $\Lambda$ is $\RR$-proper. 
\end{enumerate}
\end{prop}
\begin{proof}
\begin{enumerate}
\item We first check that $\Lambda$ is closed in $T^*\RR\times T^{*}_{\mbt} G$.
 The map  
 \[
  \phi : \RR\times T^{*}_{\mbt} G\longrightarrow\RR\times T^{*}_{\mbt} G,\quad (t,\lambda)\longmapsto (t,\chi(t,\lambda))
 \]
is a diffeomorphism and $ \RR\times \Lambda_0= \RR\times T^{*}_{\mbt} G\cap \RR\times A^*G$ is closed in $\RR\times T^{*}_{\mbt} G$ since $\RR\times A^*G$ is closed in $T^*(\RR\times G)$. Thus $\phi(\RR\times \Lambda_0)$ is closed in $\RR\times T^{*}_{\mbt} G$. It follows that 
\[
 \Lambda = \{(t,-p^{0}(\lambda),\lambda)\in T^*\RR\times T^{*}_{\mbt} G \ ;\  (t,\lambda)\in \phi(\RR\times \Lambda_0) \} 
\]
is closed in $T^*\RR\times T^{*}_{\mbt} G$. 
\item By remark \ref{rmk:prop-flow},  the inclusion $\Lambda\subset T^*\RR\times \dt{T}^*G$ holds true and by ellipticity of $P$, the function 
$p^{0} = p^{0}\circ \rg$ does not vanish on $ \dt{T}^*G$, hence 
\[
 \Lambda\subset (T^*\RR\setminus 0)\times \dt{T}^*G.
\]

\item
Since $r_{\widetilde{{}_{\Gamma}}}(t,\tau,\lambda)= (t,\rg(\lambda))$ and  $s_{\widetilde{{}_{\Gamma}}}(t,\tau,\lambda)= (t,\sg(\lambda))$, we immediately deduce 
the invertibility of $\Lambda$ from the invertibility of the $G$-relations $\Lambda_t$ for all $t$. 
\item
Since $r_{{}_{\Gamma_{+}}}(t,\tau,\lambda)= (\tau,\rg(\lambda))$,  $s_{{}_{\Gamma_{+}}}(t,\tau,\lambda)= (\tau,\sg(\lambda))$ and $ \Lambda\subset (T^*\RR\setminus 0)\times \dt{T}^*G$ we get $\Lambda\cap \ker r_{{}_{\Gamma_{+}}} = \emptyset$  and the same for $s_{{}_{\Gamma_{+}}}$ so $\Lambda$ is a $G_+$-relation. Moreover, 
denoting by $\pi,\pi_0,\pi_2$ the natural projection maps:
 \[
  \pi:  T^*(\RR\times G)\to \RR\times G,\quad  \pi_0: A^*G\to M,\quad \pi_2 : T^*(\RR\times G)\to T^*G,
 \]
 since $(t,\tau,\lambda)\in\Lambda\mapsto \sg(\lambda)\in A^*G$ is a submersion, the composition
\[
 \pi_0\circ \sg\circ \pi_2 = s_{G_+}\circ \pi|_{\Lambda} : \Lambda \longrightarrow A^*G_+\longrightarrow M, (t,\tau, \gamma,\xi)\longmapsto s(\gamma) 
\]
is a submersion. This proves that $\Lambda$ is a $G_+$-family by \cite[Remark 15 and below]{LV1}
\item This is a straightforward consequence  of the compacity of $M=G^{(0)}$, of the homogeneity of $\chi$, and of standard continuity arguments.  
\end{enumerate}
\end{proof}

\section{Approximation of $e^{-itP}$ by $G$-FIOs}\label{sec:U}

The manifold $\RR\times G$ will be provided by the pull back of the half density bundle used for $G$, and it will still be denoted by $\Omega^{1/2}$.

Let $\Lambda$ be the $\widetilde{G}$-relation defined by  $P$ as in \eqref{eq:Lambda}. 
Since $\Lambda$ is a family $\widetilde{G}$-relation,  any $U\in I^{m}(\RR\times G,\Lambda;\Omega^{1/2})$ is a Fourier integral $\widetilde{G}$-operator (see \cite{LV1} for the details),  also given as a distribution on $\widetilde{G}$  by the  $C^{\infty}$ family $U(t)\in I^{m+1/4}(G,\Lambda_t;\Omega^{1/2})$ of $G$-FIOs defined by $U(t)=i_t^*(U)$. Here $i_t :G\to \RR\times G$ is the inclusion $i_{t}(\gamma)=(t,\gamma)$. The converse is true: any such family gives a single distribution in $ I^{m}(\RR\times G,\Lambda;\Omega^{1/2})$.  
\begin{thm}\label{thm:WF_set_e^{itP}}
There exists a   Fourier integral $\widetilde{G}$-operator $U\in I^{-\frac14 +(n^{(1)} - n^{(0)})/4}(\RR\times G,\Lambda;\Omega^{1/2})$ with $\RR$-proper support such that:
\begin{equation}
 (\frac{\partial}{\partial t} + iP) U \in C^{\infty}(\RR\times G, \Omega^{1/2}).
\end{equation}
Moreover, if  $E = (e^{-itP})_{t\in\RR}$ denotes the one parameter group defined in Section \ref{sec:generalite-E-itP},  we have:
\begin{equation}\label{thm:param-E}
  E - U \in C^{\infty}(\RR, \cH^{\infty} ).
\end{equation}
\end{thm}
\begin{rmk}\ 
\begin{enumerate}
 \item It follows that $(E(t))_{t\in\RR}$ is  a $C^{\infty}$ family of distributions, equivalently $E \in \cD'_{\pr_{1}}(\RR\times G,\Omega^{1/2})$.
 \item Recall that $\cH^\infty \subset C^{\infty,0}_{\orb}(G)$ by Section \ref{sec:Sobolev}, in particular the error term \eqref{thm:param-E} is $C^{\infty}$ on  $\RR\times G_{\cO}$ for any orbit $\cO\subset M$.
 \item Theorem \ref{thm:WF_set_e^{itP}} also gives information about the operators $e^{itP_{x}}$ on the (usually non compact, complete, with bounded geometry) manifolds $G_{x}$, $x\in M$.  In the latter situation, we refer to \cite{mazzucato-nistor2006} for related results.  
\end{enumerate}
\end{rmk}
\begin{proof}[Proof of the theorem]
 Let  $U\in I^{m}(\RR\times G,\Lambda;\Omega^{1/2})$. We first check that:
 \begin{equation}\label{eq:first-error-term}
 \frac{\partial}{\partial t}U\in I^{m+1}(\RR\times G,\Lambda; \Omega^{1/2}) \text{ and } PU \in I^{m+1}(\RR\times G,\Lambda; \Omega^{1/2}). 
\end{equation}
The distribution $PU$ is given by convolution product in  $\widetilde{G}$ of the $\widetilde{G}$-PDO $P$ with the $\widetilde{G}$-FIO $U$. 
Therefore,  the composition theorem of \cite{LV1} applies and proves $ PU\in I^{m+1}(\RR\times G,\Lambda,\Omega^{1/2})$. Note that $PU$ is  also a convolution of distributions  in $G_+$:
\begin{equation}
 PU = (\delta_{t-s}\otimes P)*_{G_+}U,
\end{equation}
but this time it is not a composition of $G_+$-FIO because $\delta_{t-s}\otimes P$ fails to be in general a $G_+$-PDO.
The other assertion in \eqref{eq:first-error-term}  can be checked either by  directly differentiating with respect to $t$ the family $(U(t))_{t}$ expressed in local coordinates with oscillatory integrals, or by composing the differential $G_+$-operator $\frac{\partial}{\partial t}$ with the $G_+$-FIO $U$. 

The next task is to prove that the sum $(\frac{\partial}{\partial t} + iP)U$  is actually of order $m$ and has principal symbol given by:
\begin{equation}\label{eq:first-transport}
 \cL_{\tau + p^{0}}(u) + i \sigma^{1s}(P)u. 
\end{equation}
Since $\frac{\partial}{\partial t}+iP$ is neither a $\widetilde{G}$ nor $G_+$ pseudodifferential operator, we can not directly apply Proposition \ref{prop:adapt-25-2-4} to extract the principal symbol of \eqref{eq:first-error-term}. We propose two ways to overcome this difficulty, both containing useful technics. 

\textbf{First approach}. Both distributions $\frac{\partial}{\partial t}U $ and $PU$ are  $\widetilde{G}$-FIO. Working as before in suitable local coordinates $(t,\gamma,\xi)$, and using for instance \cite[Theorems 5 and 6]{LV1}, there exists a $C^{\infty}$ function $h(t,\xi)$, homogeneous of order $1$ in $\xi$, and a symbol $u(t,\xi)$, such that:
\begin{equation}\label{eq:h-using-fougro}
  (t,\tau, \gamma,\xi) \in  \Lambda \iff \tau =  - h'_{t}(t,\xi), \ \gamma = h'_{\xi}(t,\xi),
\end{equation}
\begin{equation}\label{eq:U-using-fougro}
   U(t,\gamma) =  \int e^{i(\langle \gamma, \xi\rangle - h(t,\xi))} u(t,\xi) d\xi.
\end{equation}
It immediately follows that 
\begin{equation}\label{eq:first-err-term-using-fougro}
(\frac{\partial}{\partial t}+iP)U(t,\gamma) =  \int e^{i(\langle \gamma, \xi\rangle - h(t,\xi))} \frac{\partial u}{\partial t}(t,\xi)d\xi + 
   \int e^{i(\langle \gamma, \xi\rangle - h(t,\xi))} i(p(\gamma,\xi)- h'_{t}(t,\xi))u(t,\xi) d\xi.
\end{equation}
The right hand side is again a sum of Lagrangian distributions. The principal symbol of the first term in the right hand side of \eqref{eq:first-err-term-using-fougro} is just the restriction to $\Lambda$ of:
\begin{equation}\label{eq:symbol-partial-t-u}
\frac{\partial u}{\partial t}= \cL_{\tau} u 
 \end{equation}
In the second term, although $ p - h'_{t}$ does not satisfy symbol estimates in $\xi$, the product $ i(p(\gamma,\xi)- h'_{t}(t,\xi))u(t,\xi)$ does  and its leading part, which is represented by $ i(p^{0} - h'_{t})u$, vanishes on $\Lambda_{t}$ for any $t$.
We then reproduce the computations starting with \eqref{eq:concise-developped-PA}, just replacing $h(\xi)$ by $h(t,\xi)$,    $q(\gamma,\xi)$ by $p(\gamma,\xi) - h'_{t}(t,\xi)$ and $a(\xi)$ by $u(t,\xi)$, without omitting an extra factor $i$.   The reminder $e$  is unchanged $e = (p - h'_{t}) - (p^{0}  - h'_{t}) = p - p^{0}$. The vector field $H_{p^{0} - h'_{t}}$ being tangent to $\Lambda_{t}$ for any $t$, we get,  since $h'_{t}$ is independent of $\gamma$:
$$ H_{p^{0} -h'_{t}}  = -\frac{\partial}{\partial \gamma_j}(p^{0} - h'_{t})\frac{\partial }{\partial \xi_j} = -\frac{\partial p^{0}}{\partial \gamma_j}\frac{\partial }{\partial \xi_j} =  H_{p^{0}}. $$
Now we can read the expression for the required principal symbol in  \eqref{al:extract-symbol}:
\begin{equation}
 \cL_{p^{0}}(u) + \left(ie - \frac{1}{2}\sum_k  \frac{\partial^2 {(p^{0}-h'_{t})}}{\partial \xi_k\partial\gamma_k}\right)u. 
\end{equation}
Again, since $h'_{t}$ is independent of $\gamma$ and $p^{0}$ independent of $\xi''$, the last expression simplifies to: 
\begin{equation}\label{eq:symbol-pu}
 \cL_{p^{0}}(u) + \left(ie - \frac{1}{2}\sum_k  \frac{\partial^2 {p^{0}}}{\partial \xi'_k\partial\gamma'_k}\right)u =  \cL_{p^{0}}(u) + i \sigma^{1s}(P)u. 
\end{equation}
Summing up  \eqref{eq:symbol-partial-t-u} and \eqref{eq:symbol-pu}, we conclude that the principal symbol of $(\frac{\partial}{\partial t}+iP)U$ is \eqref{eq:first-transport}.

\bigskip \textbf{Second approach}.  We wish to use Proposition \ref{prop:adapt-25-2-4} in the framework of the groupoid $G_{+}$.   However, we need to have the convolution of a pseudodifferential  $G_{+}$-operator  with a $G_{+}$-FIO. The problem is that the distribution $(\frac{\partial}{\partial t}+i\delta_{t-s}\otimes P)$ is not a $G_+$-pseudodifferential operator, unless $P$ is differential. The trick (similar to the one used in the proof of \cite[Theorem 25.2.4]{Horm-classics}), consists in finding a suitable microlocal approximation $\delta_{t-s}\otimes P =P_{1} + P_{2}$ of $\delta_{t-s}\otimes P $ by a $G_{+}$-PDO $P_{1}$ such that $P_{2}U\in C^{\infty} (\RR\times G)$. For that purpose, observe that we can deduce from (\ref{eq:Lambda}) that  there exists constants $c_1, c_2>0$ such for any $(\tau, x, \xi) \in r_{\Gamma_+}(\Lambda)$, we have \begin{equation} \label{tau_sur_xi}
c_1 \leq \frac{|\tau|} {|\xi|} \leq c_2. 
\end{equation}
Indeed, we know that  $(\tau, x , \xi)=r_{\Gamma_+}(t, -p(\lambda), \lambda)= (-p(\lambda), r_\Gamma (\lambda))$
for some $t \in \RR$ and $\lambda \in \Lambda_t$.
Denoting $\lambda=(\gamma , \eta) \in T^*G\setminus 0$ and $r_\Gamma (\gamma, \eta)= (x, \xi) \in A^*G\setminus 0$, we then get by homegenity of $p_0$, 
$$\frac{|\tau|} {|\xi|}= \frac{|p_0 \circ r_\Gamma(\gamma, \eta)|}{|\xi|}=\frac{|p_0(x, \xi)|}{|\xi |}=|p_0(x,\frac{\xi}{|\xi|})|$$
and the result follows by continuity of $p_0$  and compacity of $M=G^{(0)}$ (which implies that $S^*G=\{(x, \xi) \in A^*G\setminus 0, |\xi|=1\}$ is compact.)
We will use 
\begin{lem}
The distribution $\delta_{t-s}\otimes P$ on $G_+$ can be written $\delta_{t-s}\otimes P=P_1+P_2$ with $P_1$ a $G_+$-pseudodifferential operator and $P_2$ a distribution on $G_+$ such that $\WF{P_2} \subset\{(t, \tau, \lambda) \in T^*G_+\setminus 0, \sg(\lambda)=(x, \xi) \mbox{ with } \frac {|\xi|}{|\tau|} <\varepsilon\}$.
In particular the total symbol of $P_1$ and $P$  coincide in a neighborhood of $r_{\Gamma_+}(\Lambda)$ and one has that $P_2 U \in C^\infty(G_+)$ if $U \in I^{m}(\RR\times G,\Lambda,\Omega^{1/2}) .$
\end{lem}
\begin{proof}
Consider a map $\chi$ on $A^*(G_+)=  \RR_\tau \times A^*G $ such that $\chi$ is homogeneous of degree $0$ in the cotangent variables outside a compact set, and such that for a chosen $\varepsilon$ , one has \begin{enumerate}
\item $\chi(x, \xi, \tau)=0$ unless $ 1 < \varepsilon | \tau|$ and $|\xi | <\varepsilon  | \tau|$ ;
\item $\chi(x, \xi, \tau)=1$ if $ 2 < \varepsilon | \tau|$ and $2|\xi | <\varepsilon  | \tau|$ .
\end{enumerate}
If $P(x, \xi)$ is a total symbol for $P$, then one can write $$P(x, \xi)=p_1(x, \xi,\tau)+p_2(x, \xi ,\tau)=(1-\chi(x, \xi, \tau)) P(x,\xi)+\chi(x, \xi,  \tau) P(x,\xi).$$
It is clear that $p_1(x, \xi, \tau) \in S^1(A^*G_+)$, so that the corresponding operator $$P_1 (t,\gamma )=\int e^{i <\kappa(\gamma),\xi>+i<t, \tau>} p_1(r(\gamma),\xi, \tau)d\xi d\tau\, \mu^{1/2}_{s}(\gamma)\mu^{1/2}_{r}(\gamma) \in \Psi^1_{G_+}. $$ Moreover, in the neighbourhood of $r_{\Gamma_+}(\Lambda)$, one has that  $\chi(x, \xi, \tau)=0$, because of (\ref{tau_sur_xi}) and hence the symbol of $P_1$ is the symbol of $P$. \\
Now the wave front of the distribution : 
$$P_2 (t,\gamma)=\int e^{i <\kappa(\gamma),\xi> + i<t, \tau>} \chi(r(\gamma),\xi, \tau)P(r(\gamma), \xi)d\xi d\tau\, \mu^{1/2}_{s}(\gamma)\mu^{1/2}_{r}(\gamma)$$ is such that  if $(t, \tau, \lambda) \in T^*G_+\setminus 0$ and $s_\Gamma(\lambda) =(x, \xi) \in A^*G\setminus 0$ , $r_\Gamma(\lambda) =(y, \eta) \in A^*G\setminus 0$, then $(t, \tau, \lambda) \in \WF{P_2} \implies  \max(\frac {|\xi|}{|\tau|} , \frac {|\eta|}{|\tau|})\leq\varepsilon.$ This implies in particular that $\WF{P_2}. \Lambda=\emptyset.$
\end{proof}
To conclude this second approach, note that the principal symbol of $(\frac{\partial}{\partial t}+iP_1)$ is  equal to $\tau +p_0$ in a neighboorhood of $r_{\Gamma_+}(\Lambda)$ and vanishes on $r_{\Gamma_+}(\Lambda)$, because $(\tau +p_0)\circ r_{\Gamma_+}=\tau +p_0\circ r_\Gamma$ vanishes on $\Lambda$. Thus we may apply Proposition \ref{prop:adapt-25-2-4} with $G_{+}$ as underlying groupoid to the operators $(\frac{\partial}{\partial t}+iP_1)$  and $U$, which allows to recover the formula \eqref{eq:first-transport} for the principal symbol of their product by remarking that the subprincipal symbol of $P_1$ is also equal to the subprincipal of $P$ in a neighboorhood of $r_{\Gamma_+}(\Lambda)$.

The rest of the proof is essentially identical to the proof of \cite[Theorem 29.1.1]{Horm-classics}. 
Indeed the (transport) equation 
\begin{equation}
 \begin{cases}
  ( \frac{\partial}{\partial t}+ \cL_{p^{0}} + i\sigma^{1s}(P))u^0 = 0 \\
  u^0(0,.) = 1 
 \end{cases}
\end{equation}
has a unique solution $u^0\in C^\infty(\Lambda)$, and  $u^0$ homogeneous of degree $0$ with respect to the $\RR_+$ action on each $\Lambda_t$. Let us fix a $\RR$-proper set $\cV\subset \RR\times G$ such that $\supp \Lambda\subset \overset{\circ}{\cV}$. choose $U^0\in I^{(n^{(1)}-n^{(0)} - 1)/4}(\RR\times G,\Lambda,\Omega^{1/2})$ with principal symbol $u^0$ and support in $\cV$. Note that  $U^{0}(0)\in \Psi^{0}_{G,c}$ because  
$\Lambda_0=A^*G\setminus 0$. It follows that:
\begin{equation}
I- U^0(0)\in \Psi^{-1}_{G,c} \text{ and } (\frac{\partial}{\partial t}+iP)U^0 =F^1\in I^{-1+(n^{(1)}-n^{(0)} - 1)/4}(\RR\times G,\Lambda,\Omega^{1/2}).
\end{equation}
Next one chooses $U^1\in I^{-1+(n^{(1)} - n^{(0)}-1)/4}(\RR\times G,\Lambda,\Omega^{1/2})$ with support in $\cV$  and principal symbol $u^1$ solving the transport equation
\begin{equation}
 \begin{cases}
  ( \frac{\partial}{\partial t}+ \cL_{p^{0}} + i\sigma^{1s}(P))u^1 = -f^1 \\
  u^1(0,.) =   \sigma(I- U^0(0))
 \end{cases}
 \end{equation}
 and so on.  We construct in this way a sequence $U^j\in I^{- j+(n^{(1)} - n^{(0)} - 1)/4}(\RR\times G,\Lambda,\Omega^{1/2})$. Finally we choose $ U \in  I^{(n^{(1)} - n^{(0)} - 1)/4}(\RR\times G,\Lambda,\Omega^{1/2})$ with support in $\cV$ such that:
 $$
   U\sim \sum U^j.
 $$
By construction, we get  
\begin{equation}\label{eq:R-S-smoothing}
R := (\frac{\partial}{\partial t}+iP)U \in C^{\infty}(\RR\times G) \quad \text{ and } \quad S := \Id - U(0) \in C^\infty_{c}(G).
\end{equation} 
Modifying $U$ into $U +\varphi S$ with $\varphi\in C^{\infty}_{c}(\RR)$ and $\varphi(0)=1$, we can directly assume that $U(0)=\Id$. Also, the support of $R$ is contained in 
$$ \cV' = \cV \cup  (\RR\times\supp P) \cdot_{\widetilde{G}}\cV =   \cV \cup \{ (t,\gamma)\ ; \ \gamma \in \supp P \cdot \cV_{t}\}. $$ 
The set $\cV'$ is again $\RR$-proper. This implies: 
$$R\in C^{\infty}(\RR, C^{\infty}_{c}(G)) \subset C^{\infty}(\RR, \cH^{\infty} ) $$

We obtain, using (\ref{sol-cauchy-complete}) and following verbatim the proof of \cite[Theorem 29.1.1]{Horm-classics}
\begin{equation}\label{eq:U-E-smoothing}
 U(t) -e^{-itP}=\tilde{R}(t) := i\int_0^t e^{i(t-s)P}R(s) ds 
\end{equation}
Using the results of Section \ref{sec:generalite-E-itP}, we get  $\tilde{R}\in C^\infty(\RR,\cH^{\infty} )$, which ends the proof. 
\end{proof}
The previous theorem is only stated for compactly supported operators, but it admits the following slight generalization: 
\begin{cor}\label{cor:bonus}
Let $T = P_{c}+ S\in \Psi^{1}_{G}$, with $P_{c}\in \Psi^{1}_{G,c}$ satisfying the assumption of Theorem \ref{thm:WF_set_e^{itP}} and $S\in\cH^{\infty}$. There exists    $U\in I^{-\frac14 +(n^{(1)} - n^{(0)})/4}(\RR\times G,\Lambda;\Omega^{1/2})$ with $\RR$-proper support  such that 
\begin{equation}
(\frac{\partial}{\partial t} + iT)U \in  C^{\infty}(\RR, \cH^{\infty}). 
\end{equation}
\end{cor}
\begin{proof}
Apply  Theorem \ref{thm:WF_set_e^{itP}} to $P_{c}$ and let  $U\in I^{-\frac14 +(n^{(1)} - n^{(0)})/4}(\RR\times G,\Lambda;\Omega^{1/2})$  be the corresponding parametrix. Then 
\begin{equation}
(\frac{\partial}{\partial t} + iT)U = iSU + R, \qquad R \in  C^{\infty}(\RR,\cH^{\infty}). 
\end{equation}
Using the continuity theorems for $G$-FIO \cite{LV1}, one gets that for any $t$, $U_{t}$ acts continuously on the scale of Sobolev modules, which immediately implies that $SU\in C^{\infty}(\RR,\cH^{\infty})$.
\end{proof}

As examples of situations into which Theorem \ref{thm:WF_set_e^{itP}} and Corollary \ref{cor:bonus} apply, we mention:
\begin{enumerate}
\item The pair groupoid $G = X \times X \rightrightarrows X$ of a compact manifold without boundary $X$. Since $X$ itself is an orbit, we have $C^{\infty,0}_{\orb}(G)=C^{\infty}(X\times X,\Omega^{1/2})$ and we just recover the classical result (see \cite[Theorem 29.1.1]{Horm-classics} for instance), after the obvious identification between $G$-operators and continuous linear operators $C^{\infty}(X,\Omega^{1/2}_{X})\to C^{\infty}(X,\Omega^{1/2}_{X})$.
\item The holonomy groupoid $G$ of a compact foliated manifold $X$.   We recover the  construction of the leafwise geometrical optic approximation of $e^{itP}$ given in \cite{Kordyukov1994}. The latter is  worked out for small time and by solving eikonal equations to find the required phases in local coordinates as well as by solving transport equations. Our construction here can be viewed as a complement, available for arbitrary time and regarding the evolution of singularities as well as the kind of Fourier integral operators involved in the problem.
\item $G \rightrightarrows \{ e \}$ a Lie group. Here again, there is only one orbit so $C^{\infty,0}_{\orb}(G)=C^{\infty}(G)$. The result applies for instance to the square root $\sqrt{\Delta}$ of any elliptic laplacian $\Delta = - \sum X_{j}^{2}\in \Diff^{2}_{G}$, viewed as right invariant operators on $G$. That $\sqrt{\Delta} =P_{c}+S \in \Psi^{1}_{G}$ with $\sigma_{G}^{0}(P_{c})=\sqrt{\sum \xi_{j}^{2}}$ follows from \cite{Vassout2006} and we get here the existence of a $C^{\infty}$ family $U_{t}$ of right invariant $FIO$ on $G$ \cite{NielStetk1974,LV1} such that $(\frac{\partial}{\partial t} + i\sqrt{\Delta})U_{t}   \in  C^{\infty}(G)\cap C^{*}_{r}(G)$ for any $t$. 
\item The groupoid $G_{b} \rightrightarrows X$ of the $b$-calculus of a manifold with embedded corners $X$ \cite{Monthubert2003}. We recall that $G_{b}$ is the open submanifold with corners of the $b$-stretched product of R. Melrose $X^{2}_{b}$ in which all the lateral faces are removed.  Identifying $G_{b}$-operators with pseudodifferential operators in the $b$-calculus, and their restrictions at boundary hypersurfaces with indicial operators, we get for any elliptic symmetric $P\in \Psi^{1}_{b}(X)$ in the small calculus the existence of a  $C^{\infty}$ family $U_{t}$  of $b$-FIO on $X$ \cite{Melrose1981,LV1} such that 
 \begin{equation}\label{eq:b-U-int}
  (\frac{\partial}{\partial t} + iP)U_{t} = R_{t}  \in  C^{\infty}((X\setminus\partial X)^{2})\cap \cL(L^{2}_{b}(X)).
 \end{equation}
 and for any boundary hypersurfaces $H$ (with normal bundle trivialized with a boundary defining function):
 $$ I_{H}(R_{t})=  (\frac{\partial}{\partial t} + iI_{H}(P))I_{H}(U_{t})  \in  C^{\infty}(H^{2}\times \RR)\cap \cL(L^{2}_{b}(H\times \RR)). $$
The  error term $R_{t}$ is $C^{0}$ on $G_{b}$ and there is no reason neither to expect that it is $C^{\infty}$ on $G_{b}$, nor that  it  extends continuously to $X^{2}_{b}$. 
\item This discussion is similar to the previous one for the groupoid $G_{\pi} \rightrightarrows X$ \cite{DLR} and its associated pseudodifferential calculus, where $X$ is a manifold with iterated fibred corners. In both cases, the regularity result that we reach for the error term $R$ is likely not optimal. This will be investigated, among other applications to singular spaces, in future works. 
\end{enumerate}
As far as we know, examples (3--5) above are new. 
 
\bibliographystyle{plain}  
\def\cprime{$'$}

\end{document}